\documentclass{amsart}

\usepackage{amsmath, amssymb, tikz}
\tikzstyle{n}=[circle, draw, fill, minimum size=6, inner sep=0]
\tikzstyle{ngr}=[circle, draw, fill, gray, minimum size=6, inner sep=0]
\tikzstyle{int}=[draw, circle, fill,  minimum size=4, inner sep=1]
\tikzstyle{root}=[circle, draw, fill, minimum size=0, inner sep=0.7]

\tikzstyle{intgr}=[draw, circle, fill, gray,  minimum size=3, inner sep=1]
\tikzstyle{ext}=[draw, circle,  minimum size=3, inner sep=1]
\tikzset{lab/.style={draw, circle, minimum size=5, inner sep=1.1}, 
n/.style={draw, circle, fill, minimum size=5, inner sep=1}}
\usetikzlibrary{matrix}
\usetikzlibrary{arrows} 
\usetikzlibrary{decorations.pathreplacing}

\usepackage{hyperref}

\newtheorem{thm}{Theorem}[section]
\newtheorem{prop}[thm]{Proposition}

\newtheorem{lemma}[thm]{Lemma}
\newtheorem{claim}[thm]{Claim}
 
\newtheorem{remark}[thm]{Remark}

\numberwithin{equation}{section}

\numberwithin{figure}{section}

\newcommand{\dontprint}[1]{\relax}

\theoremstyle{definition}

\newcommand{\dis}{\displaystyle}

\newcommand{\br}{{\}\hspace{-0.07cm}.\hspace{-0.03cm}.\hspace{-0.07cm}\} }}

\newcommand{\leftbr}{{\{\hspace{-0.07cm}.\hspace{-0.03cm}.\hspace{-0.07cm}\{ }}

\newcommand{\Graphs}{\mathsf{Graphs}} 

\newcommand{\Br}{{\mathsf{Br} }} 

\newcommand{\Com}{{\mathsf{Com} }}
\newcommand{\coCom}{{\mathsf{coCom} }}
\newcommand{\Lie}{{\mathsf {Lie}}}
\newcommand{\As}{{\mathsf {As}}}
\newcommand{\coAs}{{\mathsf{coAs}}}
\newcommand{\coLie}{{\mathsf{coLie}}}
\newcommand{\Ger}{{\mathsf{Ger}}}

\newcommand{\Cbu}{C^{\bullet}}

\newcommand{\ve}{{\varepsilon}}

\newcommand{\mj}{{\mathfrak{j}}}

\newcommand{\md}{{\mathfrak{d}}}

\newcommand{\bfzero}{{\bf 0}}

\newcommand{\id}{{\mathrm {i d} }}

\newcommand{\sgn}{\mathrm{sgn}}

\newcommand{\La}{{\Lambda}}

\newcommand{\la}{{\lambda}}

\newcommand{\de}{{\delta}}

\newcommand{\si}{{\sigma}}

\newcommand{\pa}{{\partial}}

\renewcommand{\c}{{\circ}}
\newcommand{\bul}{{\bullet}}

\newcommand{\bs}{{\bf s}}
\newcommand{\bsi}{{\bf s}^{-1}\,}

\newcommand{\opp}{\mathrm{opp}}

\newcommand{\Ind}{\mathsf{Ind}}

\newcommand{\cA}{\mathcal{A}}
\newcommand{\cT}{\mathcal{T}}

\newcommand{\cM}{\mathcal{M}}
\newcommand{\cO}{\mathcal{O}}

\newcommand{\cF}{\mathcal{F}}
\newcommand{\cC}{\mathcal{C}}

\newcommand{\bbK}{{\mathbb{K}}}
\newcommand{\bbQ}{{\mathbb{Q}}}
\newcommand{\bbZ}{{\mathbb{Z}}}

\newcommand{\Sh}{\mathrm{Sh}}

\DeclareMathOperator{\Gr}{Gr}

\DeclareMathOperator{\coker}{coker}

\DeclareMathOperator{\vsspan}{span}

\title{A direct computation of the cohomology of the braces operad}

\author{Vasily Dolgushev}
\author{Thomas Willwacher}

\begin{document}

\large

\date{}

\begin{abstract}
We give a self-contained and purely combinatorial proof of the well known fact that 
the cohomology of the braces operad is the operad $\Ger$ governing 
Gerstenhaber algebras. 
\end{abstract}

\maketitle

~\\[-0.5cm]
{\small 
Deligne conjecture on Hochschild cochains.} \\
~\\

\section{Introduction}
It is a well known fact \cite{Ger} 
that the Hochschild cohomology of an associative (or $A_\infty$) algebra $A$ carries the structure 
of a Gerstenhaber algebra. In 1993, P. Deligne \cite{Deligne} asked whether this Gerstenhaber algebra 
structure is induced by an action of some version of the chains operad of the little disks operad on 
the Hochschild cochain complex $\Cbu(A,A)$ of $A$.
This question became known as the Deligne conjecture and was answered affirmatively by various authors 
including C. Berger and B. Fresse \cite{BF},  R.M. Kaufmann \cite{Kauf-Spineless}, \cite{Kauf-Cacti},
M. Kontsevich and Y. Soibelman \cite{K-Soi}, J. E. McClure and J. H. Smith \cite{M-Smith}, and
D. Tamarkin \cite{Dima-another}, \cite{Dima-dg}.

A key role in the proof of the Deligne conjecture is played by the \emph{braces operad} $\Br$, which encodes 
a set of natural operations on the Hochschild cochain complex $\Cbu(A,A)$ of any $A_\infty$ algebra $A$.
In the form used here, this differential graded (dg) operad was 
introduced by Kontsevich and Soibelman \cite{K-Soi}, where it is called the ``minimal operad''. Its (quasi-isomorphic) variant\footnote{See also \cite[\S3.12]{Kauf-moduli} for the description of the map between the asscociative and $A_\infty$ version.} for associative algebras 
was considered by McClure and Smith, \cite{M-Smith2}, and both constructions go back to earlier work of Getzler \cite{Get} (cf. also \cite{GJ}).

The goal of this note is to give a purely combinatorial proof of the fact that the operad $H^{\bul}(\Br)$ is isomorphic to the operad 
$\Ger$ which governs Gerstenhaber algebras. 

Concretely, the $n$-th space $\Br(n)$ of the braces operad is 
spanned\footnote{In this note, the ground field $\bbK$ is any field of characteristic zero.} by 
planar rooted trees (called {\it brace trees}) with $n$ vertices labeled by $1,2, \dots, n$ and some (possibly zero) number 
of unlabeled (or {\it neutral}) vertices. The grading on $\Br(n)$ is obtained by declaring that 
each non-root edge carries degree $-1$ and each neural vertex carries degree $2$. 
In pictures, white circles with inscribed numbers denote labeled vertices and black 
circles denote neutral vertices\footnote{We tacitly assume that every neutral vertex 
has at least two children.}. Several examples of brace trees are shown 
in figure \ref{fig:examples}.
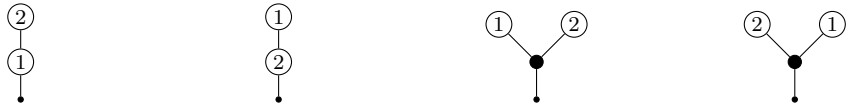
\begin{figure}[htp] 
\centering 
\begin{minipage}[t]{0.2\linewidth}
\centering
\begin{tikzpicture}[ baseline=-.65ex]
\node[root] (r) at (0, -0.5) {};
\draw (0,0) node[lab] (v1) { $\scriptstyle 1$};
\draw (0,0.6) node[lab] (v2) { $\scriptstyle 2$};
\draw (v1) edge (v2) edge (r);
\end{tikzpicture}
\end{minipage}
~
\begin{minipage}[t]{0.2\linewidth}
\centering
\begin{tikzpicture}[ baseline=-.65ex]
\node[root] (r) at (0, -0.5) {};
\draw (0,0) node[lab] (v2) { $\scriptstyle 2$};
\draw (0,0.6) node[lab] (v1) { $\scriptstyle 1$};
\draw (v2) edge (v1) edge (r);
\end{tikzpicture}
\end{minipage}
~
\begin{minipage}[t]{0.2\linewidth}
\centering
\begin{tikzpicture}[ baseline=-.65ex]
\node[root] (r) at (0, -0.5) {};
\draw (0,0) node[n] (n) {};
\draw (-0.5,0.5) node[lab] (v1) { $\scriptstyle 1$};
\draw (0.5,0.5) node[lab] (v2) { $\scriptstyle 2$};
\draw (n) edge (v1) edge (v2) edge (r);
\end{tikzpicture}
\end{minipage}
~
\begin{minipage}[t]{0.2\linewidth}
\centering
\begin{tikzpicture}[ baseline=-.65ex]
\node[root] (r) at (0, -0.5) {};
\draw (0,0) node[n] (n) {};
\draw (-0.5,0.5) node[lab] (v2) { $\scriptstyle 2$};
\draw (0.5,0.5) node[lab] (v1) { $\scriptstyle 1$};
\draw (n) edge (v1) edge (v2) edge (r);
\end{tikzpicture}
\end{minipage}
\caption{The brace trees $T_{1\mbox{-}2}$,  $T_{2\mbox{-}1}$, $T_{\cup}$ and 
$T_{\cup}^{\opp}$ from left to right, respectively} \label{fig:examples}
\end{figure} 
Thus the brace trees $T_{1\mbox{-}2}$ and $T_{2\mbox{-}1}$ 
have degree $-1$ while the brace trees $T_{\cup}$ and 
$T_{\cup}^{\opp}$ have degree $0$.  
 
The differential $\de(T)$ of a brace tree $T$ is defined by the formula 
$$
\de(T) : =\sum_{j=1}^n  \de_j(T)  ~ + ~ \sum_{v} \de_v (T), 
$$ 
where the second sum is over all neutral vertices and the operations 
$\de_j$, $\de_v$ are defined graphically as follows:
\begin{align*}
\de_j 
\begin{tikzpicture}[baseline=-.65ex]
\draw (0,0) node[lab] (v1){$\scriptstyle j$} ;
\draw (v1) edge (.4,.5) edge (.8,.5) edge (-.8,.5) edge (-.4,.5) edge (0,.5) edge (0,-.5);
\end{tikzpicture}
&=
\sum~
\pm
\begin{tikzpicture}[baseline=-.65ex]
 \draw (0,0) node[lab] (v1) { $\scriptstyle j$} edge (.8,.5) edge (-.8,.5) edge (0,-.5)
      (0,.5) node[n] (v2) {} edge (v1) edge (.4,1) edge (-.4,1) edge (0,1) ;
\end{tikzpicture}
\quad + \quad
\sum~
\pm
\begin{tikzpicture}[baseline=-2ex]
 \draw (0,0) node[lab] (v1) { $\scriptstyle j$} edge (.4,.5) edge (-.4,.5) edge (0,.5)
       (0,-.5) node[n] (v2) {} edge (.8,0) edge (-.8,0) edge (0,-1) edge (v1)  ;
\end{tikzpicture}
& \qquad&
\delta_v
\begin{tikzpicture}[ baseline=-.65ex]
 \draw (0,0) node[n] (v1) {} edge (.4,.5) edge (.8,.5) edge (-.8,.5) edge (-.4,.5) edge (0,.5) edge (0,-.5);
\end{tikzpicture}
&=
\sum~
\pm
\begin{tikzpicture}[ baseline=-.65ex,yshift=-.25cm]
 \draw (0,0) node[n] (v1) {} edge (.8,.5) edge (-.8,.5) edge (0,-.5)
       (0,.5) node[n] (v2) {} edge (v1) edge (.4,1) edge (-.4,1) edge (0,1);
\end{tikzpicture}
\end{align*}

The operadic multiplications are defined in terms of natural combinatorial operations with 
planar trees. For more details we refer the reader to Section 3 of this note or 
\cite[Sections 7-9]{DW}.  

The dg operad $\Br$ acts on the Hochschild cochain complex $\Cbu(A,A)$ of an 
$A_{\infty}$-algebra. The detailed description of this action is given in \cite[Appendix B]{DW}. 
For example, for $P_1, P_2 \in \Cbu(A,A)$, the cochain $T_{\cup} (P_1, P_2)$  
(resp. $T_{1\mbox{-}2} (P_1,P_2) + T_{2\mbox{-}1} (P_1,P_2) $) coincides  
(up to a sign factor) with the cup product $P_1 \cup P_2$ (resp. the Gerstenhaber bracket 
$[P_1, P_2]_G$).   

Let us recall (see Appendix \ref{app:Ger-relation}) that the $S_2$-invariant $\de$-cocycle 
\begin{equation}
\label{T-brack}
T_{\{a_1, a_2 \}} : = T_{1\mbox{-}2} +  T_{2\mbox{-}1}
\end{equation}
satisfies the Jacobi relation 
$$
T_{\{a_1, a_2 \}} \circ_1 T_{\{a_1, a_2 \}} + (1,2,3) \big( T_{\{a_1, a_2 \}} \circ_1 T_{\{a_1, a_2 \}}  \big) + 
(1,3,2) \big( T_{\{a_1, a_2 \}} \circ_1 T_{\{a_1, a_2 \}} \big) = 0. 
$$  
Therefore, we have a natural operad map 
\begin{equation}
\label{mj}
\mj : \La\Lie \to \Br
\end{equation}
from the shifted version $\La\Lie$ of the operad $\Lie$ to the dg operad $\Br$. 

It is easy to check that the cocycle $T_{\cup}$ satisfies the associativity relation up to homotopy
$$
T_{\cup} \circ_1 T_{\cup} -  T_{\cup} \circ_2 T_{\cup} \in \textrm{Im}(\de)
$$
and the difference 
$$
T_{\cup}- T^{\opp}_{\cup}
$$
is $\de$-exact. 

Therefore, we have a natural operad map 
\begin{equation}
\label{Com-to-HBr}
\Com \to H^{\bul}(\Br) 
\end{equation}
which sends the generator of $\Com$ to the cohomology class of the $\de$-cocycle: 
\begin{equation}
\label{T-a1a2}
T_{a_1 a_2} : = \frac{1}{2}(T_{\cup} + T^{\opp}_{\cup}).
\end{equation}

It is also easy to check (see Appendix \ref{app:Ger-relation}) that the $\de$-cocycles $T_{a_1 a_2}$ and $T_{\{a_1, a_2\}}$ 
satisfy the Leibniz rule up to homotopy, i.e. 
$$
T_{\{a_1, a_2\}} \circ_2 T_{a_1 a_2} - T_{a_1 a_2} \circ_1 T_{\{a_1, a_2\}} - 
(1,2) \big( T_{a_1 a_2} \circ_2 T_{\{a_1, a_2\}} \big)  \in \textrm{Im}(\de). 
$$ 

Thus, combining the maps \eqref{mj} and \eqref{Com-to-HBr}, we get an operad map 
\begin{equation}
\label{Ger-to-HBr}
\Ger \to H^{\bul}(\Br).  
\end{equation}

In this note, we give a self-contained combinatorial proof of the following theorem:
\begin{thm}
\label{thm:main}
The map \eqref{Ger-to-HBr} is an isomorphism of operads. 
\end{thm}

This theorem is a shadow of the very deep statement which says that the dg operad $\Br$ 
is weakly equivalent to the operad $\Ger$. The proof of the latter statement involves a solution of the Deligne 
conjecture and the formality of the dg operad $C_{-\bul}(E_2, \bbK)$ where $E_2$ denotes the topological 
operad of little discs \cite{Dima-disc}. One possible proof  \cite{Dima-disc} of the formality of $C_{-\bul}(E_2, \bbK)$ 
involves the use of Drinfeld's associator \cite{Drinfeld} and another possible proof \cite[Section 3.3]{K-mot}, 
\cite{LV-formality} involves the use of a configuration space integral. Although Theorem \ref{thm:main} does 
not imply the formality of the operad $\Br$, it is amazing that it can be proved in a purely combinatorial way
which bypasses the use of compactified configuration spaces.


We should remark that various topological proofs of Theorem \ref{thm:main} were given earlier. 
One such proof is sketched, for example, in \cite[Theorem 4]{K-Soi}, and another proof 
may be extracted from \cite{M-Smith2}, together with a small computation. 
Finally a third proof is described in \cite{Kauf-Cacti, Kauf-Schwell}.

%
%

%

Let us also remark that our proof admits a straightforward generalization to the higher versions of 
the braces operads $\Br_{n+1}$ acting naturally on the deformation complexes of $n$-algebras, cf. \cite[Section 4]{CaWi}.

\begin{remark}
\label{rem:Br-versus-Graphs}
There is an amazing combinatorial similarity between the dg operad $\Br$ 
and the dg operad $\Graphs$  \cite[Section 3.3]{K-mot}, \cite[Section 3]{SW}.
The latter dg operad is ``assembled from'' graphs of certain kind with some 
additional data and the former dg operad is ``assembled from'' rooted planar trees 
(also with some additional data). Both dg operads are formal. In fact, both 
dg operads are weakly equivalent to the same operad $\Ger$. 
However, while the proof of formality for $\Graphs$ involves only elementary 
homological algebra  \cite[Section 3.3.4]{K-mot}, the proof of formality for $\Br$ requires 
a ``very heavy hammer''.     
\end{remark}

\subsection{The organization of the paper and the outline of the proof of Theorem \ref{thm:main}}
\label{sec:outline}
In Section \ref{sec:notation}, we fix some necessary notational conventions. In Section
\ref{sec:Br-desc}, we give a more detailed description of the dg operad $\Br$. 
In Section \ref{sec:H-Br}, we formulate and prove a more refined version of  Theorem \ref{thm:main} 
(see Theorem \ref{thm:H-Br}). Appendix \ref{app:Ger-relation} is devoted to the proof of 
the fact that the vectors \eqref{T-brack} and \eqref{T-a1a2} of $\Br$ satisfy the Gerstenhaber relations 
up to homotopy. Finally, Appendix \ref{app:E2-Einfty} is devoted to a proof of a technical statement 
about the spectral sequence used in Section \ref{sec:H-Br}.

Using a standard basis for the space $\Ger(n)$ and vectors \eqref{T-brack} and \eqref{T-a1a2}, 
we define a map of dg collections $\Psi : \Ger \to \Br$ (see Section \ref{sec:Psi}). 

Claim \ref{cl:Ger-relations} from Appendix \ref{app:Ger-relation} implies that the map 
$$
H^{\bul} (\Psi) : \Ger \to H^{\bul}(\Br) 
$$
is compatible with the operad structure. 

To prove that the map $H^{\bul} (\Psi)$ induces an isomorphism of operads, 
we proceed by induction on the arity $n$. 

Since the base of the induction $n=1$ is obvious, we assume that $\Psi$ induces isomorphisms 
$H^{\bul}(\Br(j)) \cong \Ger(j)$ for all $1 \le j \le n-1$, we split the graded vector 
space into the direct sum 
\[
\begin{tikzpicture}
\matrix (m) [matrix of math nodes]
{
\Br(n) &= & V_\circ(n) &\oplus & V_\bullet(n)\,, \\ 
};
\draw (m-1-3) edge[bend left, ->] node[above] {$\delta_1$} (m-1-5);
\draw (m-1-3) edge[loop above] node[above] {$\delta_0$} (m-1-3);
\draw (m-1-5) edge[loop above] node[above] {$\delta_0$} (m-1-5);
\end{tikzpicture}
\]
where $V_\bullet(n)$ is the subspace of $\Br(n)$ spanned by brace trees whose lowest non-root vertex
is neutral and $V_\circ(n)$ is the subspace of $\Br(n)$ spanned by brace trees whose lowest non-root vertex 
is labeled. The arrows in the above formula indicate the non-zero components of the differential.

It is clear that 
\begin{itemize}

\item both $V_\circ(n)$ and $V_{\bul}(n)$ may be considered as 
cochain complexes with the differential $\de_0$; 

\item $\de_1$ induces a map 
\begin{equation}
\label{main-map}
H^{\bul}(\delta_1) ~:~
H^{\bul}(V_\circ(n), \delta_0) \longrightarrow H^{\bul}(V_\bullet(n), \delta_0)\,; 
\end{equation}

\item and, finally, 
$$
H^{\bul}(\Br(n)) \cong \big( \ker H^{\bul}(\delta_1) \big) \oplus \big( \coker H^{\bul}(\delta_1) \big).
$$

\end{itemize}

In Section \ref{sec:H-V-circ}, we prove that $H^{\bul}(V_\circ(n), \delta_0)$ is isomorphic to
$\bs^{n-1} \bbK[S_n]$ as the $S_n$-module and show that the cohomology class corresponding 
to $\la \in S_n$ is represented by the brace tree $T^n_{\la}$ depicted in figure \ref{fig:T-la}.

In Section \ref{sec:H-V-bul}, we establish an isomorphism of $S_n$-modules
\begin{equation}
\label{isom-HV-bullet}
H^{\bul}(V_\bullet(n), \delta_0) \cong   \Com \odot \La \Lie (n) \big/ \La \Lie (n)  ~\oplus ~
\bs \big( \La\Com \odot \La \Lie (n) \big/ \La \Lie (n) \big),
\end{equation}
where $\odot$ denotes the plethysm of collections. 

This is done by filtering $V_{\bul}(n)$ by the number of children of the lowest non-root vertex
and analyzing the corresponding spectral sequence. The main technical statement 
$$
E_{\infty} (V_\bullet(n), \delta_0) = E_{2} (V_\bullet(n), \delta_0)
$$ 
about this spectral sequence is proved separately (see Lemma \ref{lem:specseqabutment}) 
in Appendix \ref{app:E2-Einfty}.


In Section \ref{sec:H-de-1}, we prove a technical statement about the dual version of 
the map \eqref{main-map}. Finally, in Section \ref{sec:the-end}, we use this technical statement 
and the results of the previous sections to complete the proof of Theorem \ref{thm:H-Br}. 

~\\[-0.5cm]

\noindent
\textbf{Acknowledgements:} V.D. has been partially supported by NSF grants DMS-1161867 and 
DMS-1501001 and T.W. has been partially supported by the Swiss National Science foundation, 
grant 200021\_150012, and the SwissMAP NCCR funded by the Swiss National Science foundation.
We would like to thank anonymous referees for their useful comments and suggestions. 

\section{Notation}
\label{sec:notation}

We work over a ground field $\bbK$ of characteristic 0. 
For a set $X$ we denote by $\vsspan_{\bbK}(X)$ the $\bbK$-vector 
space of finite linear combinations of elements in $X$.
We denote by $\bs$ (resp. $\bs^{-1}$) the operation of suspension (resp. desuspension)
for graded or differential graded (\emph{dg} for short) $\bbK$ vector spaces.  
The notation $|v|$ is reserved for the degree of a homogeneous vector $v$ in 
a (differential) graded vector space. 

By a {\it collection} we mean a sequence $\{ P(n) \}_{n \ge 0}$ 
of dg vector spaces with a right action of the symmetric group $S_n$\,.  
The category of collections carries a natural monoidal structure, the plethysm operation $\odot$, see, e.~g., \cite[eqn. (5.1)]{DW}.

We will freely use the language of operads. A good introduction is provided in textbook \cite{Loday-Vallette}.
The notation $\Lie$ (resp. $\As$, $\Com$, $\Ger$) is used for the 
operad governing Lie algebras (resp. associative, commutative or Gerstenhaber algebras without unit).
Dually, the notation $\coLie$ (resp. $\coAs$, $\coCom$) is reserved for the 
cooperad governing Lie coalgebras (resp.  coassociative coalgebras without 
counit, cocommutative (and coassociative)
coalgebras without counit).

For an operad $\cO$ (resp. a cooperad $\cC$) and a cochain complex $V$, we denote by 
$\cO(V)$ (resp. $\cC(V)$) the free $\cO$-algebra (resp. cofree $\cC$-coalgebra). 

For an operad (resp. a cooperad) $P$ we denote by $\La P$ the 
operad (resp. the cooperad) with the spaces of $n$-ary operations: 
\begin{equation}
\label{La-P}
\La P (n) = \bs^{1-n} P(n) \otimes \sgn_n\,,
\end{equation}
where $\sgn_n$ denotes the sign representation of $S_n$\,.

For an operad $\cO$ and degree $0$ auxiliary variables 
$ a_1, a_2, \dots, a_n $, 
$\cO(n)$ is naturally identified with the subspace of the free 
$\cO$-algebra 
$$
\cO\big( \vsspan_{\bbK}(a_1, a_2, \dots, a_n) \big) 
$$
spanned by $\cO$-monomials in which each variable from the set 
$\{ a_1, a_2, \dots, a_n \} $ 
appears exactly once. We often use this identification 
in this paper. For example, the vector space $\Ger(2)$ of the operad $\Ger$ is spanned 
by the degree zero vector $a_1 a_2$ and the degree $-1$ vector $\{a_1, a_2\}$\,.
The commutative (and associative) multiplication on a Gerstenhaber algebra $V$
comes from the vector $a_1 a_2 \in \Ger(2)$ and the odd Lie bracket 
$\{~,~\}$ on $V$ comes from the vector $\{a_1, a_2\} \in \Ger(2)$. 
Similarly, the space $\La\Lie(n)$ of the suboperad $\La\Lie \subset \Ger$ is 
spanned by $\La\Lie$-monomials in $a_1, a_2, \dots, a_n$ in which each 
variable  from the set $\{ a_1, a_2, \dots, a_n \} $ appears exactly once.
For example,  $\La\Lie(2)$ is spanned by the vector $\{a_1, a_2\}$ and 
$\La\Lie(3)$ is spanned by the vectors $\{\{a_1, a_2\}, a_3\}$ and 
$\{\{a_1, a_3\}, a_2\}$.

Let us recall \cite[Section 2, p. 32]{notes} that the set of edges of any planar tree $T$ is equipped with 
the natural total order. We use this total order to determine sign factors in various computations 
related to the operad $\Br$. 

\section{Brace trees, a reminder of the dg operad $\Br$}
\label{sec:Br-desc}
Let us recall that a \emph{brace tree} is a rooted planar tree having two kinds of non-root vertices:
\begin{itemize}
\item \emph{labeled} vertices, numbered $\{1,2,3, \dots \}$,
\item an arbitrary number of unlabeled \emph{neutral} vertices. 
\end{itemize}
In addition, one requires that each neutral vertex has at least two children. 
For example, figure \ref{fig:exam} shows a brace tree $T$ with $6$ labeled vertices.
In pictures, white circles with inscribed numbers denote labeled vertices, black 
circles denote neutral vertices, and the small black node (at the bottom) denotes the root.  
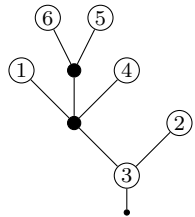
\begin{figure}[htp] 
\centering 
\begin{tikzpicture}
[lab/.style={draw, circle, minimum size=5, inner sep=1}, 
n/.style={draw, circle, fill, minimum size=5, inner sep=1}, scale=.7]
\draw (1,0) node[lab] (v3) {$\scriptstyle 3$} 
(0,1)  node[n] (n1) {} 
(2,1) node[lab] (v2) {$\scriptstyle 2$} 
(1,-0.7) node[root] (r) {};
\draw (-1,2) node[lab] (v1) {$\scriptstyle 1$}
(0, 2) node[n] (n2) {}  (1,2) node[lab] (v4) {$\scriptstyle 4$};
\draw (-0.5,3) node[lab] (v6) {$\scriptstyle 6$}
(0.5,3) node[lab] (v5) {$\scriptstyle 5$};
\draw (v3) edge (v2) edge (n1) edge (r) 
(n1) edge (v1) edge (v4) edge (n2) (n2) edge (v6) edge (v5);
\end{tikzpicture}
\caption{An example of a brace tree} \label{fig:exam}
\end{figure}
 
$\Br(n)$ is the linear span of the set of brace trees with exactly $n$ labeled vertices.
The $\bbZ$-grading on $\Br(n)$ is given by declaring that each brace tree has 
the degree
$$
2  \times \#\textrm{ of neutral vertices }~  - ~ \# \textrm{ of non-root edges}.
$$ 
For example, the brace tree shown in figure \ref{fig:exam} has degree $-3$.  

Let $T$ be a brace tree with $n$ labeled vertices, $j$ be a number between 
$1$ and $n$, and $v$ be a neutral vertex of $T$ (if $T$ has one). 
To recall the definition of the differential $\de$ on $\Br(n)$, we introduce these 
three vectors
$$
\de'_j (T), \qquad  \de''_j (T), \quad \textrm{and} \quad \de_v(T)
$$ 
in $\Br(n)$. The vector $\de'_j (T)$ (resp. $\de''_j(T)$) is obtained from $T$ in the three steps: 
\begin{itemize}

\item first, we replace vertex $j$ by the left most branch in figure \ref{fig:branches} (resp. the middle brach in 
figure \ref{fig:branches}); 

\item  second, we reconnect the edges which originated from vertex $j$ to 
this branch in all ways compatible with the planar structure; 

\item finally, we discard all brace trees which have a neutral vertex of valency $< 3$.  

\end{itemize}
  
Similarly, the vector $\de_v(T)$ is obtained from $T$ in the three steps: 
\begin{itemize}

\item first, we replace the neutral vertex $v$ with the right most branch in figure \ref{fig:branches}; 

\item second, we reconnect the edges which originated from vertex $v$ to 
this branch in all ways compatible with the planar structure;

\item finally, we discard all brace trees which have a neutral vertex of valency $< 3$.  

\end{itemize}
\begin{figure}[htp] 
\centering 
\begin{minipage}[t]{0.3\linewidth}
\centering
\begin{tikzpicture}[ baseline=-.65ex]
\draw (0,0) node[n] (n) {~};
\draw (0,0.6) node[lab] (j) { $\scriptstyle j$};
\draw (n) edge (j) edge (0,-0.4);
\end{tikzpicture}
\end{minipage}
~
\begin{minipage}[t]{0.3\linewidth}
\centering
\begin{tikzpicture}[ baseline=-.65ex]
\draw (0,0.6) node[n] (n) {~};
\draw (0,0) node[lab] (j) { $\scriptstyle j$};
\draw (j) edge (n) edge (0,-0.4);
\end{tikzpicture}
\end{minipage}
~
\begin{minipage}[t]{0.3\linewidth}
\centering
\begin{tikzpicture}[ baseline=-.65ex]
\draw (0,0) node[n] (n1) {~};
\draw (0,0.6) node[n] (n2) {~};
\draw (n1) edge (n2) edge (0,-0.4);
\end{tikzpicture}
\end{minipage}
\caption{The branches appearing in the definition of the differential} \label{fig:branches}
\end{figure}
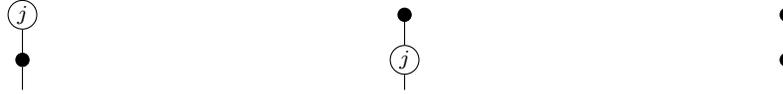 

The differential $\de(T)$ of a brace tree $T \in \Br(n)$ is the sum 
over all labeled and all neutral vertices 
$$
\de(T) = \sum_{j=1}^n (\de'_j(T) + \de''_j(T)) ~ + ~ \sum_{v} \de_v (T). 
$$

The signs in the sums $\de'_j(T)$, $\de''_j(T)$, and $\de_v(T)$ are determined by treating non-root edges 
as ``anti-commuting variables.''

%
%
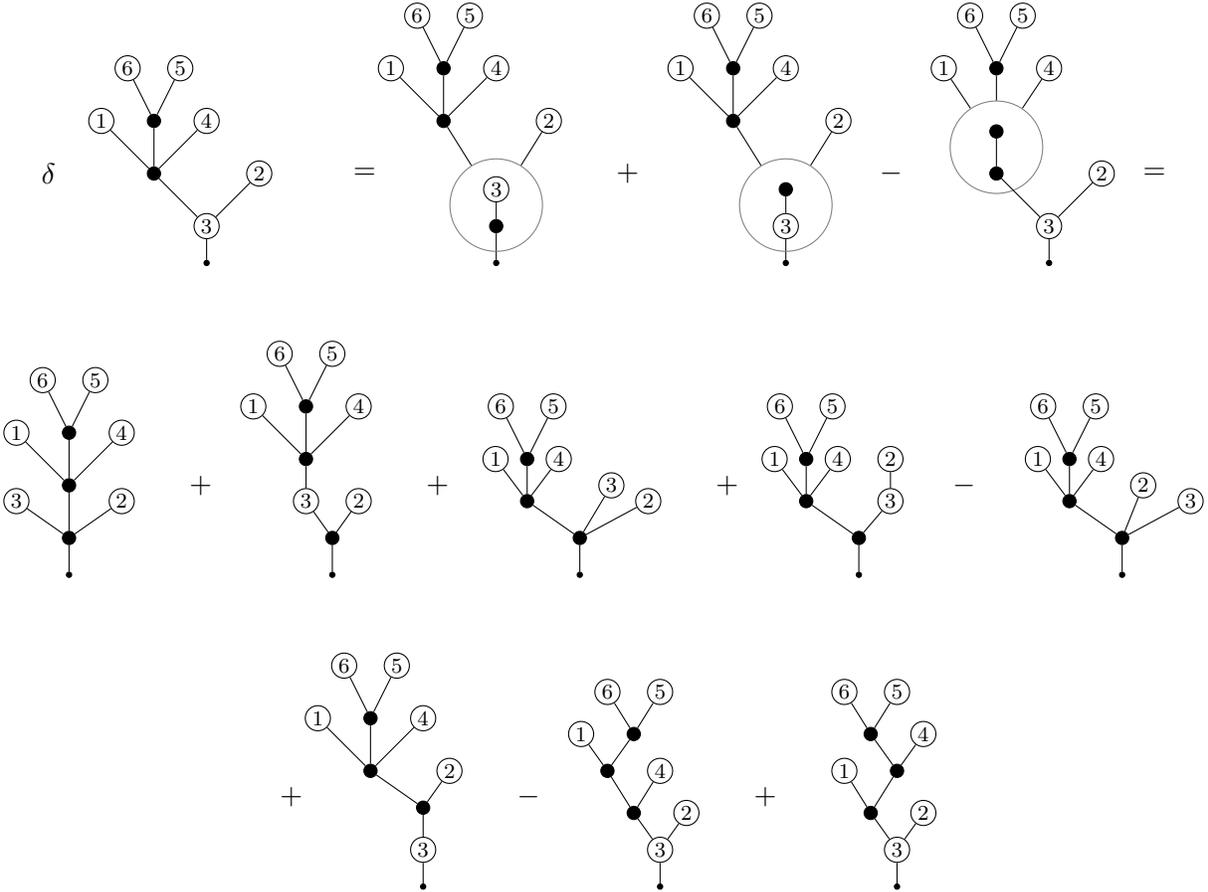
\begin{figure}[htp] 
\centering 
\begin{minipage}[t]{\linewidth}
\centering 
\begin{tikzpicture}
[lab/.style={draw, circle, minimum size=5, inner sep=1}, 
n/.style={draw, circle, fill, minimum size=5, inner sep=1}, scale=.7]
\draw  (-2,1) node { $\de$};
\draw (1,0) node[lab] (v3) {$\scriptstyle 3$} 
(0,1)  node[n] (n1) {} 
(2,1) node[lab] (v2) {$\scriptstyle 2$} 
(1,-0.7) node[root] (r) {};
\draw (-1,2) node[lab] (v1) {$\scriptstyle 1$}
(0, 2) node[n] (n2) {}  (1,2) node[lab] (v4) {$\scriptstyle 4$};
\draw (-0.5,3) node[lab] (v6) {$\scriptstyle 6$}
(0.5,3) node[lab] (v5) {$\scriptstyle 5$};
\draw (v3) edge (v2) edge (n1) edge (r) 
(n1) edge (v1) edge (v4) edge (n2) (n2) edge (v6) edge (v5);
\draw  (4,1) node { $=$};
\begin{scope}[shift={(5.5,0)}]
\draw (1,0.7) node[lab] (vv3) {$\scriptstyle 3$} (1,0)  node[n] (nn1) {} (1,-0.7) node[root] (rr) {};
\draw (nn1) edge (vv3) edge (rr); 
\draw (1,0.4) node[draw, color=gray, circle, minimum size=35, inner sep=1] (big) {~}; 
\draw (0,2) node[n] (nn2) {} (2,2) node[lab] (vv2) {$\scriptstyle 2$} ;
\draw (-1,3) node[lab] (vv1) {$\scriptstyle 1$} (0, 3) node[n] (nn3) {}  (1,3) node[lab] (vv4) {$\scriptstyle 4$};
\draw (-0.5,4) node[lab] (vv6) {$\scriptstyle 6$} (0.5,4) node[lab] (vv5) {$\scriptstyle 5$};
\draw (big) edge (nn2) edge (vv2)  (nn2) edge (vv1) edge (nn3) edge (vv4) (nn3) edge (vv6) edge (vv5);
\end{scope}
\draw  (9,1) node { $+$};
\begin{scope}[shift={(11,0)}]
\draw (1,0) node[lab] (vv3) {$\scriptstyle 3$} (1,0.7)  node[n] (nn1) {} (1,-0.7) node[root] (rr) {};
\draw (vv3) edge (nn1) edge (rr); 
\draw (1,0.4) node[draw, color=gray, circle, minimum size=35, inner sep=1] (big) {~}; 
\draw (0,2) node[n] (nn2) {} (2,2) node[lab] (vv2) {$\scriptstyle 2$} ;
\draw (-1,3) node[lab] (vv1) {$\scriptstyle 1$} (0, 3) node[n] (nn3) {}  (1,3) node[lab] (vv4) {$\scriptstyle 4$};
\draw (-0.5,4) node[lab] (vv6) {$\scriptstyle 6$} (0.5,4) node[lab] (vv5) {$\scriptstyle 5$};
\draw (big) edge (nn2) edge (vv2)  (nn2) edge (vv1) edge (nn3) edge (vv4) (nn3) edge (vv6) edge (vv5);
\end{scope}
\draw  (14,1) node { $-$};
\begin{scope}[shift={(16,0)}]
\draw (1,0) node[lab] (v3) {$\scriptstyle 3$} 
(0,1)  node[n] (n1) {} 
(2,1) node[lab] (v2) {$\scriptstyle 2$} 
(1,-0.7) node[root] (r) {};
\draw (-1,3) node[lab] (v1) {$\scriptstyle 1$}
(0, 1.8) node[n] (n2) {}  (0, 3) node[n] (n3) {}  
(1,3) node[lab] (v4) {$\scriptstyle 4$};
\draw (-0.5,4) node[lab] (v6) {$\scriptstyle 6$}
(0.5,4) node[lab] (v5) {$\scriptstyle 5$};
\draw (0,1.5) node[draw, color=gray, circle, minimum size=35, inner sep=1] (big) {~}; 
\draw (v3) edge (v2) edge (n1) edge (r) 
(n1) edge (n2)  (n3) edge (v6) edge (v5)
(big) edge (v1) edge (v4) edge (n3);
\end{scope}
\draw  (19,1) node { $=$};
\end{tikzpicture}
\end{minipage}
\begin{minipage}[t]{\linewidth}
\vspace{0.5cm}
\end{minipage}
\begin{minipage}[t]{\linewidth}
\centering 
\begin{tikzpicture}
[lab/.style={draw, circle, minimum size=5, inner sep=1}, 
n/.style={draw, circle, fill, minimum size=5, inner sep=1}, scale=.7]
\draw (1,0)  node[n] (n1) {} (1,-0.7) node[root] (rr) {};
\draw (0,0.7) node[lab] (v3) {$\scriptstyle 3$};
\draw (1,1) node[n] (n2) {}; 
\draw (2,0.7) node[lab] (v2) {$\scriptstyle 2$};
\draw (0,2) node[lab] (v1) {$\scriptstyle 1$};
\draw (1, 2) node[n] (n3) {} ; 
\draw (2,2) node[lab] (v4) {$\scriptstyle 4$};
\draw (0.5,3) node[lab] (v6) {$\scriptstyle 6$}; 
\draw (1.5,3) node[lab] (v5) {$\scriptstyle 5$};
\draw (n1) edge (v3) edge (rr) edge (n2) edge (v2) (n2) edge (v1) edge (n3) edge (v4) (n3) edge (v6) edge (v5);
\draw  (3.5,1) node { $+$};
\begin{scope}[shift={(5,0)}]
\draw (0.5,0.7) node[lab] (v3) {$\scriptstyle 3$}; 
\draw (1,0)  node[n] (n1) {} (1,-0.7) node[root] (rr) {};
\draw (1.5,0.7) node[lab] (v2) {$\scriptstyle 2$};
\draw (0.5,1.5) node[n] (n2) {};
\draw (-0.5,2.5) node[lab] (v1) {$\scriptstyle 1$};
\draw (0.5, 2.5) node[n] (n3) {}; 
\draw (1.5,2.5) node[lab] (v4) {$\scriptstyle 4$};
\draw (0,3.5) node[lab] (v6) {$\scriptstyle 6$}; 
\draw (1,3.5) node[lab] (v5) {$\scriptstyle 5$};
\draw (n1) edge (rr) edge (v3) edge (v2);
\draw (n2) edge (v3) edge (v1) edge (n3) edge (v4);
\draw (n3) edge (v5) edge (v6);
\end{scope}
\draw  (8,1) node { $+$};
\begin{scope}[shift={(9.7,0)}]
\draw (1,0)  node[n] (n1) {} (1,-0.7) node[root] (rr) {};
\draw (1.6,1) node[lab] (v3) {$\scriptstyle 3$}; 
\draw (2.3,0.7) node[lab] (v2) {$\scriptstyle 2$};
\draw (0,0.7) node[n] (n2) {};
\draw (-0.6,1.5) node[lab] (v1) {$\scriptstyle 1$};
\draw (0, 1.5) node[n] (n3) {}; 
\draw (0.6,1.5) node[lab] (v4) {$\scriptstyle 4$};
\draw (-0.5,2.5) node[lab] (v6) {$\scriptstyle 6$}; 
\draw (0.5,2.5) node[lab] (v5) {$\scriptstyle 5$};
\draw (n1) edge (rr) edge (v3) edge (v2) edge (n2);
\draw (n2) edge (v1) edge (n3) edge (v4);
\draw (n3) edge (v5) edge (v6);
\end{scope}
\draw  (13.5,1) node { $+$};
\begin{scope}[shift={(15,0)}]
\draw (1,0)  node[n] (n1) {} (1,-0.7) node[root] (rr) {};
\draw (1.6,0.7) node[lab] (v3) {$\scriptstyle 3$}; 
\draw (1.6,1.5) node[lab] (v2) {$\scriptstyle 2$};
\draw (0,0.7) node[n] (n2) {};
\draw (-0.6,1.5) node[lab] (v1) {$\scriptstyle 1$};
\draw (0, 1.5) node[n] (n3) {}; 
\draw (0.6,1.5) node[lab] (v4) {$\scriptstyle 4$};
\draw (-0.5,2.5) node[lab] (v6) {$\scriptstyle 6$}; 
\draw (0.5,2.5) node[lab] (v5) {$\scriptstyle 5$};
\draw (n1) edge (rr) edge (v3) edge (n2) (v3) edge (v2);
\draw (n2) edge (v1) edge (n3) edge (v4);
\draw (n3) edge (v5) edge (v6);
\end{scope}
\draw  (18,1) node { $-$};
\begin{scope}[shift={(20,0)}]
\draw (1,0)  node[n] (n1) {} (1,-0.7) node[root] (rr) {};
\draw (1.4,1) node[lab] (v2) {$\scriptstyle 2$};
\draw (2.3,0.7) node[lab] (v3) {$\scriptstyle 3$}; 
\draw (0,0.7) node[n] (n2) {};
\draw (-0.6,1.5) node[lab] (v1) {$\scriptstyle 1$};
\draw (0, 1.5) node[n] (n3) {}; 
\draw (0.6,1.5) node[lab] (v4) {$\scriptstyle 4$};
\draw (-0.5,2.5) node[lab] (v6) {$\scriptstyle 6$}; 
\draw (0.5,2.5) node[lab] (v5) {$\scriptstyle 5$};
\draw (n1) edge (rr) edge (v3) edge (n2) edge (v2);
\draw (n2) edge (v1) edge (n3) edge (v4);
\draw (n3) edge (v5) edge (v6);
\end{scope}
\end{tikzpicture}
\end{minipage}
\begin{minipage}[t]{\linewidth}
\vspace{0.5cm}
\end{minipage}
\begin{minipage}[t]{\linewidth}
\centering 
\begin{tikzpicture}
[lab/.style={draw, circle, minimum size=5, inner sep=1}, 
n/.style={draw, circle, fill, minimum size=5, inner sep=1}, scale=.7]
\draw  (-1.5,1) node { $+$};
\draw (1,0) node[lab] (v3) {$\scriptstyle 3$} (1,-0.7) node[root] (rr) {};
\draw (1,0.8)  node[n] (n1) {};
\draw (0,1.5) node[n] (n2) {}; 
\draw (1.5,1.5) node[lab] (v2) {$\scriptstyle 2$};
\draw (-1,2.5) node[lab] (v1) {$\scriptstyle 1$};
\draw (0, 2.5) node[n] (n3) {} ; 
\draw (1,2.5) node[lab] (v4) {$\scriptstyle 4$};
\draw (-0.5,3.5) node[lab] (v6) {$\scriptstyle 6$}; 
\draw (0.5,3.5) node[lab] (v5) {$\scriptstyle 5$};
\draw (v3) edge (rr) edge (n1) (n1) edge (n2) edge (v2) (n2) edge (v1) edge (n3) edge (v4) (n3) edge (v6) edge (v5);
\draw  (3,1) node { $-$};
\begin{scope}[shift={(4.5,0)}]
\draw (1,0) node[lab] (v3) {$\scriptstyle 3$} (1,-0.7) node[root] (rr) {};
\draw (1.5,0.7) node[lab] (v2) {$\scriptstyle 2$};
\draw (0.5,0.7) node[n] (n1) {}; 
\draw (0, 1.5) node[n] (n2) {};
\draw (1,1.5) node[lab] (v4) {$\scriptstyle 4$};
\draw (-0.5,2.2) node[lab] (v1) {$\scriptstyle 1$};
\draw (0.5, 2.2) node[n] (n3) {};
\draw (0,3) node[lab] (v6) {$\scriptstyle 6$};
\draw (1,3) node[lab] (v5) {$\scriptstyle 5$};
\draw (v3) edge (rr) edge (v2) edge (n1) (n1) edge (n2) edge (v4) (n2) edge (v1) edge (n3) (n3) edge (v6) edge (v5) ;
\end{scope}
\draw  (7.5,1) node { $+$};
\begin{scope}[shift={(9,0)}]
\draw (1,0) node[lab] (v3) {$\scriptstyle 3$} (1,-0.7) node[root] (rr) {};
\draw (1.5,0.7) node[lab] (v2) {$\scriptstyle 2$};
\draw (0.5,0.7) node[n] (n1) {}; 
\draw (0,1.5) node[lab] (v1) {$\scriptstyle 1$};
\draw (1,1.5) node[n] (n2) {};  
\draw (1.5, 2.2) node[lab] (v4) {$\scriptstyle 4$};
\draw (0.5,2.2) node[n] (n3) {};
\draw (0, 3) node[lab] (v6) {$\scriptstyle 6$};
\draw (1, 3) node[lab] (v5) {$\scriptstyle 5$};
\draw (v3) edge (rr) edge (v2) edge (n1)  (n1) edge (v1) edge (n2) (n2) edge (n3) edge (v4) (n3) edge (v6) edge (v5);
\end{scope}
\end{tikzpicture}
\end{minipage}
\caption{Example of computing $\de(T)$} \label{fig:exam-diff}
\end{figure}

For example, for the brace tree $T$ shown in figure \ref{fig:exam}, the computation of the differential 
is shown in figure \ref{fig:exam-diff}. The sign\footnote{We should remark that the differential $\pa$ defined in 
eq. (8.12) of \cite{DW} differs from $\de$ by the overall sign factor: $\de = -\pa$.} 
``$-$'' in front of the right most term in the first line 
appears to due to the fact the additional edge has to ``move behind''  the edge originating 
from vertex $3$. 
The signs in front of the first four terms in the second line are pluses since 
the brach which originates at vertex $3$ of $T$ has the even number of edges. The sign ``$-$'' in 
front of the right most term in the second line appears because the edge adjacent to vertex $2$ ``moves 
ahead'' of the additional edge. 
The signs in the third line are obtained in the similar fashion.

Let us observe that, since we discard brace trees with at least one neutral vertex of valency $\le 2$, we have
$$
\de'_j(T) = \de''_j(T) = 0 \qquad \textrm{and} \qquad \de_v(T) = 0
$$
if vertex $j$ is univalent and (neutral) vertex $v$ is trivalent.  
Also, if vertex $j$ is bivalent, then $\de''_j(T) = 0$.

%
%
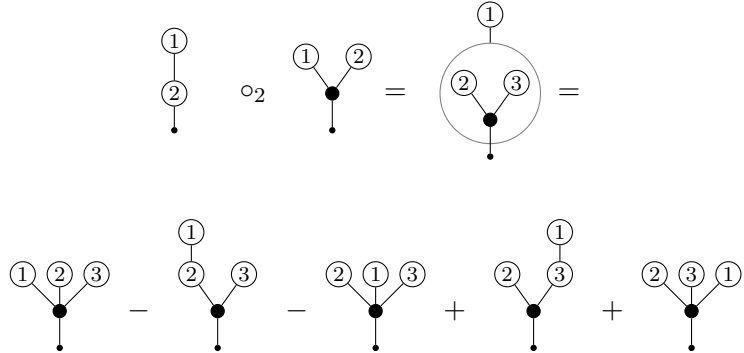
\begin{figure}[htp] 
\centering 
\begin{minipage}[t]{\linewidth}
\centering 
\begin{tikzpicture}
[lab/.style={draw, circle, minimum size=5, inner sep=1}, 
n/.style={draw, circle, fill, minimum size=5, inner sep=1}, scale=.7]
\draw (0,0) node[lab] (v2) {$\scriptstyle 2$} (0,-0.7) node[root] (rr) {};
\draw (0,1) node[lab] (v1) {$\scriptstyle 1$};
\draw (v2) edge (rr) edge (v1);
\draw  (1.5,0) node { $\circ_2$};
\begin{scope}[shift={(3,0)}]
\draw (0,0) node[n] (n) {} (0,-0.7) node[root] (rr) {};
\draw (-0.5,0.7) node[lab] (v1) {$\scriptstyle 1$};
\draw (0.5,0.7) node[lab] (v2) {$\scriptstyle 2$};
\draw (n) edge (rr) edge (v1) edge (v2);
\end{scope}
\draw  (4.2,0) node { $ = $};
\begin{scope}[shift={(6,-0.5)}]
\draw (0,0) node[n] (n) {} (0,-0.7) node[root] (rr) {};
\draw (-0.5,0.7) node[lab] (v2) {$\scriptstyle 2$};
\draw (0.5,0.7) node[lab] (v3) {$\scriptstyle 3$};
\draw (0,0.5) node[draw, color=gray, circle, minimum size=38, inner sep=1] (big) {~}; 
\draw (0,2) node[lab] (v1) {$\scriptstyle 1$};
\draw (big) edge (v1) (n) edge (rr) edge (v2) edge (v3);
\end{scope}
\draw  (7.5,0) node { $ = $};
\end{tikzpicture}
\end{minipage}
\begin{minipage}[t]{\linewidth}
\vspace{0.3cm}
\end{minipage}
\begin{minipage}[t]{\linewidth}
\centering 
\begin{tikzpicture}
[lab/.style={draw, circle, minimum size=5, inner sep=1}, 
n/.style={draw, circle, fill, minimum size=5, inner sep=1}, scale=.7]
\draw (0,0) node[n] (n) {} (0,-0.7) node[root] (rr) {};
\draw (-0.7,0.7) node[lab] (v1) {$\scriptstyle 1$};
\draw (0,0.7) node[lab] (v2) {$\scriptstyle 2$};
\draw (0.7,0.7) node[lab] (v3) {$\scriptstyle 3$};
\draw (n) edge (rr) edge (v1) edge (v2) edge (v3);
\draw  (1.5,0) node { $-$};
\begin{scope}[shift={(3,0)}]
\draw (0,0) node[n] (n) {} (0,-0.7) node[root] (rr) {};
\draw (-0.5,0.7) node[lab] (v2) {$\scriptstyle 2$};
\draw (0.5,0.7) node[lab] (v3) {$\scriptstyle 3$};
\draw (-0.5,1.5) node[lab] (v1) {$\scriptstyle 1$};
\draw (n) edge (rr) edge (v2) edge (v3) (v2) edge (v1);
\end{scope}
\draw  (4.5,0) node { $-$};
\begin{scope}[shift={(6,0)}]
\draw (0,0) node[n] (n) {} (0,-0.7) node[root] (rr) {};
\draw (-0.7,0.7) node[lab] (v2) {$\scriptstyle 2$};
\draw (0,0.7) node[lab] (v1) {$\scriptstyle 1$};
\draw (0.7,0.7) node[lab] (v3) {$\scriptstyle 3$};
\draw (n) edge (rr) edge (v1) edge (v2) edge (v3);
\end{scope}
\draw  (7.5,0) node { $+$};
\begin{scope}[shift={(9,0)}]
\draw (0,0) node[n] (n) {} (0,-0.7) node[root] (rr) {};
\draw (-0.5,0.7) node[lab] (v2) {$\scriptstyle 2$};
\draw (0.5,0.7) node[lab] (v3) {$\scriptstyle 3$};
\draw (0.5,1.5) node[lab] (v1) {$\scriptstyle 1$};
\draw (n) edge (rr) edge (v2) edge (v3) (v3) edge (v1);
\end{scope}
\draw  (10.5,0) node { $+$};
\begin{scope}[shift={(12,0)}]
\draw (0,0) node[n] (n) {} (0,-0.7) node[root] (rr) {};
\draw (-0.7,0.7) node[lab] (v2) {$\scriptstyle 2$};
\draw (0,0.7) node[lab] (v3) {$\scriptstyle 3$};
\draw (0.7,0.7) node[lab] (v1) {$\scriptstyle 1$};
\draw (n) edge (rr) edge (v1) edge (v2) edge (v3);
\end{scope}
\end{tikzpicture}
\end{minipage}
\caption{A computation of an elementary insertion} \label{fig:insert}
\end{figure}

A simple example of the computation of an elementary insertion is shown in figure \ref{fig:insert}. The sign ``$-$''
in front of the second term and the third term appears since the edge adjacent to vertex $1$ has to ``move behind'' 
the edge connecting vertex $2$ to the only neutral vertex. In the last two terms, the edge adjacent to vertex $1$ has to
``move behind'' the two edges originating from the only neutral vertex. This is why we have pluses in front of these terms. 
For the precise definition of  the operadic compositions in $\Br$, we refer the reader to \cite[Sections 7-9]{DW}.

%

\subsection{Remarks on the linear dual of $\Br$}
\label{sec:Br-dual}

Let us observe that the linear dual $\Br(n)^*$ can be canonically identified
with $\Br(n)$ as the vector space. The only difference is that 
the degree of a brace tree in $\Br(n)^*$ equals $\#$ of non-root edges $-$ $2 \times \#$ of neutral vertices. 
Using this observation, we will often switch back and forth between various subspaces  
of $\Br(n)$ (with certain differentials) and their linear duals. 

For example, the differential $\de^*$ on the dual complex $\Br(n)^*$ is the sum (with appropriate signs)
$$
\de^*(T) : = \sum_{e \in \mathrm{Edges}_{\bul}(T)} \pm \de^*_e (T), 
$$
where the brace tree $ \de^*_e (T)$ is obtained from $T$ by contracting the edge $e$ 
and the set $\mathrm{Edges}_{\bul}(T)$ consists of non-root edges $e$ which satisfy this property:
{\it $e$ either connects two neutral vertices or $e$ is adjacent to one neutral vertex.}
For instance, for the brace trees shown in figure \ref{fig:examples}, we have 
$$
\de^* (T_{\cup}) = T_{1\mbox{-}2}  -  T_{2\mbox{-}1}\,.
$$
On the other hand, if $T$ is any brace tree without neutral vertices then $\de^*(T) = 0$. 

\section{Computation of the cohomology of \texorpdfstring{$\Br$}{Br}}
\label{sec:H-Br}

\subsection{The map of collections of dg vector spaces $\Psi : \Ger \to \Br$}
\label{sec:Psi}

Let us recall (see Appendix \ref{app:Ger-relation}) that the assignment 
$$
\mj(\{a_1, a_2\}) : = T_{\{a_1, a_2\}}
$$
gives us the map of dg operad 
\begin{equation}
\label{mj-here}
\mj : \La\Lie \to \Br,
\end{equation}
where $\La\Lie$ is considered with the zero differential. 

We will use $\mj$ to define a map $\Psi$ of collections
\begin{equation}
\label{Psi}
\Psi : \Ger \to \Br. 
\end{equation}
For this purpose, we recall \cite[Exercise 3.12]{notes} that $\Ger(n)$ has the basis 
formed by the monomials 
\begin{equation}
\label{Ger-n-basis}
\{ a_{i_{11}},  \dots, \{ a_{i_{1 (p_1-1)}}, a_{i_{1 p_1}} \br \dots
\{ a_{i_{t1}},  \dots, \{ a_{i_{t (p_t-1)}}, a_{i_{t p_t}} \br,
\end{equation}
where  
\begin{equation}
\label{sp-partition}
\{i_{11}, i_{12}, \dots, i_{1 p_1}\} \sqcup 
\{i_{21}, i_{22}, \dots, i_{2 p_2}\} \sqcup \dots \sqcup \{i_{t1}, i_{t 2}, \dots, i_{t p_t}\}
\end{equation}
are ordered partitions of the set $\{1, 2, \dots, n\}$ satisfying the following properties: 
\begin{itemize}

\item for each $1 \le \beta \le t$ the index $i_{\beta p_{\beta}}$ is 
the biggest among $i_{\beta 1}, \dots, i_{\beta p_{\beta} }$

\item $i_{1 p_1}  <  i_{2 p_2} < \dots <  i_{t p_t}$ (in particular, $i_{t p_t} = n$).

\end{itemize}

Let $\si$ be the permutation in $S_n$ 
$$
\si = 
\left(
\begin{array}{cccccccccccccc}
1 & 2 & \dots & p_1 & p_1+1 & p_1+2 & \dots &  p_1+ p_2 & \dots & \dots &
n-p_t+1 &  n- p_t+2 & \dots & n  \\
i_{11} & i_{12} & \dots & i_{1 p_1} & i_{21} & i_{22} & \dots & i_{2 p_2} & \dots & \dots &
i_{t1} & i_{t 2} & \dots & i_{t p_t}
\end{array}
\right)
$$
corresponding to such a partition \eqref{sp-partition}. 

Then, for the corresponding monomial \eqref{Ger-n-basis} in the above basis,
we set\footnote{Here, we assume that $t \ge 2$.} 
\begin{equation}
\label{Psi-dfn}
\Psi(\{ a_{i_{11}},  \dots, \{ a_{i_{1 (p_1-1)}}, a_{i_{1 p_1}} \br \dots
\{ a_{i_{t1}},  \dots, \{ a_{i_{t (p_t-1)}}, a_{i_{t p_t}} \br) : = 
\end{equation}
$$
\si\big( \Psi( \{ a_{1},  \dots, \{ a_{p_1-1}, a_{p_1} \br  \{ a_{p_1+1},  \dots, \{ a_{p_1 + p_2-1 }, a_{p_1 + p_2} \br \dots
\{ a_{n-p_t+1},  \dots, \{ a_{n-1}, a_{n} \br) \big),
$$

$$
 \Psi( \{ a_{1},  \dots, \{ a_{p_1-1}, a_{p_1} \br  \{ a_{p_1+1},  \dots, \{ a_{p_1 + p_2-1 }, a_{p_1 + p_2} \br\dots
\{ a_{n-p_t+1},  \dots, \{ a_{n-1}, a_{n} \br) : =
$$
$$
\mu \big( \cM_{t};  \mj( \{ a_{1},  \dots, \{ a_{p_1-1}, a_{p_1} \br), \mj (\{ a_{1},  \dots, \{ a_{p_2-1 }, a_{p_2} \br),
\dots, \mj(\{ a_{1},  \dots, \{ a_{p_t-1}, a_{p_t} \br)  \big),   
$$
where $\mu$ is the operadic multiplication $\Br(t) \otimes (\Br(p_1) \otimes \Br(p_2) \otimes \dots \otimes \Br(p_t)) 
\to \Br(p_1+ p_2 + \dots + p_t)$, and $\cM_{t}$ is the vector 
\begin{equation}
\label{cM-t}
\cM_t : = \underbrace{(..(T_{a_1 a_2} \circ_1 T_{a_1 a_2}) \circ_1  T_{a_1 a_2}) \dots \circ_1 T_{a_1 a_2})}_{\circ_1 ~~\textrm{appears }t-2 \textrm{ times}} \in \Br(t). 
\end{equation}

Finally, if $t =1$, i.e. we deal with a monomial $v \in \La\Lie(n)$, then we set 
\begin{equation}
\label{Psi-dfn-Lie}
\Psi(v) : =  \mj(v).
\end{equation}

For example, the vector $\cM_3 =  T_{a_1 a_2} \circ_1 T_{a_1 a_2}$ is shown in figure \ref{fig:cM3}
and the vector $\Psi(a_1 a_2 \{a_3, a_4\}) \in \Br(4)$ is shown in figure \ref{fig:12-brack34}.  
%
%
\begin{figure}[htp] 
\centering 
\begin{tikzpicture}
[lab/.style={draw, circle, minimum size=5, inner sep=1}, 
n/.style={draw, circle, fill, minimum size=5, inner sep=1}, scale=.7]
\draw  (-3,0) node { $\cM_3~~ =  $};
\draw  (-1.5,0) node { $\displaystyle \frac{1}{4}$};
\draw (0,0) node[n] (n1) { } (0,-0.7) node[root] (rr) {};
\draw (-0.5,0.5) node[n] (n2) {};
\draw (0.5,0.5) node[lab] (v3) {$\scriptstyle 3$};
\draw (-1,1) node[lab] (v1) {$\scriptstyle 1$};
\draw (0,1) node[lab] (v2) {$\scriptstyle 2$};
\draw (n1) edge (rr) edge (n2) edge (v3) (n2) edge (v1) edge (v2);
\draw  (1.5,0) node { $+$};
\begin{scope}[shift={(4,0)}]
\draw  (-1.5,0) node { $\displaystyle \frac{1}{4}$};
\draw (0,0) node[n] (n1) { } (0,-0.7) node[root] (rr) {};
\draw (0.5,0.5) node[n] (n2) {};
\draw (-0.5,0.5) node[lab] (v3) {$\scriptstyle 3$};
\draw (0,1) node[lab] (v1) {$\scriptstyle 1$};
\draw (1,1) node[lab] (v2) {$\scriptstyle 2$};
\draw (n1) edge (rr) edge (n2) edge (v3) (n2) edge (v1) edge (v2);
\end{scope}
\draw  (5.5,0) node { $+$};
\begin{scope}[shift={(8,0)}]
\draw  (-1.5,0) node { $\displaystyle \frac{1}{4}$};
\draw (0,0) node[n] (n1) { } (0,-0.7) node[root] (rr) {};
\draw (-0.5,0.5) node[n] (n2) {};
\draw (0.5,0.5) node[lab] (v3) {$\scriptstyle 3$};
\draw (-1,1) node[lab] (v2) {$\scriptstyle 2$};
\draw (0,1) node[lab] (v1) {$\scriptstyle 1$};
\draw (n1) edge (rr) edge (n2) edge (v3) (n2) edge (v1) edge (v2);
\end{scope}
\draw  (9.5,0) node { $+$};
\begin{scope}[shift={(12,0)}]
\draw  (-1.5,0) node { $\displaystyle \frac{1}{4}$};
\draw (0,0) node[n] (n1) { } (0,-0.7) node[root] (rr) {};
\draw (0.5,0.5) node[n] (n2) {};
\draw (-0.5,0.5) node[lab] (v3) {$\scriptstyle 3$};
\draw (0,1) node[lab] (v2) {$\scriptstyle 2$};
\draw (1,1) node[lab] (v1) {$\scriptstyle 1$};
\draw (n1) edge (rr) edge (n2) edge (v3) (n2) edge (v1) edge (v2);
\end{scope}
\end{tikzpicture}
\caption{The vector $\cM_3 \in \Br(3)$} \label{fig:cM3}
\end{figure}
%
%
\begin{figure}[htp] 
\centering 
\begin{tikzpicture}
[lab/.style={draw, circle, minimum size=5, inner sep=1}, 
n/.style={draw, circle, fill, minimum size=5, inner sep=1}, scale=.7]
\draw  (-1.5,0) node { $\displaystyle \frac{1}{4}$};
\draw (0,0) node[n] (n1) { } (0,-0.7) node[root] (rr) {};
\draw (-0.5,0.5) node[n] (n2) {};
\draw (0.5,0.5) node[lab] (v3) {$\scriptstyle 3$};
\draw (1,1) node[lab] (v4) {$\scriptstyle 4$};
\draw (-1,1) node[lab] (v1) {$\scriptstyle 1$};
\draw (0,1) node[lab] (v2) {$\scriptstyle 2$};
\draw (n1) edge (rr) edge (n2) edge (v3) (n2) edge (v1) edge (v2) (v3) edge (v4);
\draw  (1.5,0) node { $-$};
\begin{scope}[shift={(4,0)}]
\draw  (-1.5,0) node { $\displaystyle \frac{1}{4}$};
\draw (0,0) node[n] (n1) { } (0,-0.7) node[root] (rr) {};
\draw (0.5,0.5) node[n] (n2) {};
\draw (-0.5,0.5) node[lab] (v3) {$\scriptstyle 3$};
\draw (-1,1) node[lab] (v4) {$\scriptstyle 4$};
\draw (0,1) node[lab] (v1) {$\scriptstyle 1$};
\draw (1,1) node[lab] (v2) {$\scriptstyle 2$};
\draw (n1) edge (rr) edge (n2) edge (v3) (n2) edge (v1) edge (v2) (v3) edge (v4);
\end{scope}
\draw  (5.5,0) node { $+$};
\begin{scope}[shift={(8,0)}]
\draw  (-1.5,0) node { $\displaystyle \frac{1}{4}$};
\draw (0,0) node[n] (n1) { } (0,-0.7) node[root] (rr) {};
\draw (-0.5,0.5) node[n] (n2) {};
\draw (0.5,0.5) node[lab] (v3) {$\scriptstyle 3$};
\draw (1,1) node[lab] (v4) {$\scriptstyle 4$};
\draw (-1,1) node[lab] (v2) {$\scriptstyle 2$};
\draw (0,1) node[lab] (v1) {$\scriptstyle 1$};
\draw (n1) edge (rr) edge (n2) edge (v3) (n2) edge (v1) edge (v2) (v3) edge (v4);
\end{scope}
\draw  (9.5,0) node { $-$};
\begin{scope}[shift={(12,0)}]
\draw  (-1.5,0) node { $\displaystyle \frac{1}{4}$};
\draw (0,0) node[n] (n1) { } (0,-0.7) node[root] (rr) {};
\draw (0.5,0.5) node[n] (n2) {};
\draw (-0.5,0.5) node[lab] (v3) {$\scriptstyle 3$};
\draw (-1,1) node[lab] (v4) {$\scriptstyle 4$};
\draw (0,1) node[lab] (v2) {$\scriptstyle 2$};
\draw (1,1) node[lab] (v1) {$\scriptstyle 1$};
\draw (n1) edge (rr) edge (n2) edge (v3) (n2) edge (v1) edge (v2) (v3) edge (v4);
\end{scope}
\draw  (15,0) node { $+ \quad (3 \leftrightarrow 4)$};
\end{tikzpicture}
\caption{The vector $\Psi(a_1 a_2 \{a_3, a_4\}) \in \Br(4)$} \label{fig:12-brack34}
\end{figure}

We claim that 
\begin{prop}
\label{prop:Psi-is-great!}
Equations \eqref{Psi-dfn} and \eqref{Psi-dfn-Lie} define a map of collections
of dg vector spaces 
$$
\Psi : \Ger \to \Br. 
$$
Furthermore, the induced map 
\begin{equation}
\label{H-Psi}
H^{\bul}(\Psi) : \Ger \to H^{\bul}(\Br) 
\end{equation}
is compatible with the operadic multiplications.
\end{prop}
\begin{proof}
The first statement follows from the fact that the vectors $T_{a_1 a_2}, ~ T_{\{a_1, a_2\}} \in \Br(2)$ 
are $\de$-cocycles. The second statement follows from Claim \ref{cl:Ger-relations} proved in 
Appendix \ref{app:Ger-relation}.  
\end{proof}

\subsection{The refinement of Theorem \ref{thm:main}}
\label{sec:main-st-t}

We will prove the following refined version of Theorem \ref{thm:main}:
%
%
\begin{thm}
\label{thm:H-Br}
The map of dg collections $\Psi$ defined above
induces an isomorphism of graded operads 
\begin{equation}
\label{H-Br-is-Ger}
H^{\bul}(\Br) \cong \Ger.
\end{equation}
\end{thm}

We will prove that $\Psi$ induces an isomorphism $H^{\bul}(\Br(n)) = \Ger(n)$ by induction on $n$. 

For $n=1$ there is nothing to show. 
So suppose we know that $H^{\bul}(\Br(j)) = \Ger(j)$ for $j=1,2\dots,n-1$ and let us tackle the statement for $j=n$. As outlined in the introduction, we split
\[
\begin{tikzpicture}
\matrix (m) [matrix of math nodes]
{
\Br(n) &= & V_\circ(n) &\oplus & V_\bullet(n), \\ 
};
\draw (m-1-3) edge[bend left, ->] node[above] {$\delta_1$} (m-1-5);
\draw (m-1-3) edge[loop above] node[above] {$\delta_0$} (m-1-3);
\draw (m-1-5) edge[loop above] node[above] {$\delta_0$} (m-1-5);
\end{tikzpicture}
\]
where $V_\bullet(n)$ is the subspace of $\Br(n)$ spanned by brace trees whose lowest non-root vertex
is neutral, while $V_\circ(n)$ is the subspace of $\Br(n)$ spanned by brace trees whose lowest non-root vertex is labeled. 
Again as mentioned before, we then find that 
\begin{equation}
\label{H-Br-ker-coker}
H^{\bul}(\Br(n)) = \big( \ker H^{\bul}(\delta_1) \big) \oplus \big( \coker H^{\bul}(\delta_1) \big),
\end{equation}
where 
\begin{equation}
\label{H-de-1}
H^{\bul}(\delta_1) ~:~
H^{\bul}(V_\circ(n), \delta_0) \longrightarrow H^{\bul}(V_\bullet(n), \delta_0)\,.
\end{equation}
is the induced map on $\delta_0$-cohomologies.

\subsection{Computing $H^{\bul}(V_\circ(n), \delta_0)$}
\label{sec:H-V-circ}

Following remarks in Subsection \ref{sec:Br-dual}, we begin by computing $H^{\bul}(V^*_\circ(n), \delta^*_0)$
for the dual of the complex $(V_\circ(n), \delta_0)$:
\begin{claim}
\label{cl:2}
We claim that 
\begin{equation}
\label{cl2-eq}
H^{\bul}(V_\circ^*(n), \delta_0^*)\cong \bs^{n-1}\bbK^{n!}\cong \bs^{n-1}\bbK[S_n]
\end{equation}
as $S_n$-modules. Moreover, the class corresponding to 
a permutation $\la \in S_n$ is represented by the brace tree $T^n_{\la}$ shown in figure \ref{fig:T-la}.
\end{claim}
\begin{proof}
We proceed by induction on $n$. For $n=1$ the statement is clear. 
Otherwise split:
\[
\begin{tikzpicture}
\matrix (m) [matrix of math nodes]
{
V_\circ^* &= & W_1 &\oplus & W_{\geq 2} \\ 
};
\draw (m-1-3) edge[bend left, ->] node[above] {$\delta_1'$} (m-1-5);
\draw (m-1-3) edge[loop above] node[above] {$\delta_0'$} (m-1-3);
\draw (m-1-5) edge[loop above] node[above] {$\delta_0'$} (m-1-5);
\end{tikzpicture}
\]
Here $W_1$ is spanned by brace trees in which the lowest non-root vertex has exactly one child and $W_{\geq 2}$ 
is spanned by brace trees in which the lowest non-root vertex has at least two children. 
It is easy to see that $\delta_1'$ is surjective and that its kernel is spanned 
by brace trees whose lowest non-root vertex has a labeled vertex as a child. 
The complex $(\ker \delta_1', \delta_0')$ is isomorphic 
to $(V_\circ^*(n-1), \delta^*_0)$\,. 

Thus the induction hypothesis implies that $H^{\bul}(V_\circ^*(n), \delta_0^*)  \cong  \bs^{n-1}\bbK[S_n]$
as graded vector spaces. The compatibility of the resulting isomorphism with the $S_n$-action is obvious. 
\end{proof}

\begin{remark}
\label{rem:TTprime}
Recall that every brace tree $T \in \Br(n)^*$ without neutral vertices is automatically $\de^*$-closed 
and hence $\delta^*_0$-closed. Therefore, by Claim \ref{cl:2}, for every brace tree $T \in \Br(n)^*$
without neutral vertices, there exists a vector $T' \in V_\circ^*(n)$, 
such that $T-\delta_0^*T'$ is a linear combination of string-like brace trees, i.e. brace trees of the form 
$T^n_{\la}$ (see figure \ref{fig:T-la}). 
\end{remark}

\subsection{Computing $H^{\bul}(V_\bullet(n), \delta_0)$}
\label{sec:H-V-bul}

To compute $H^{\bul}(V_\bullet(n), \delta_0)$,
we filter the cochain complex $(V_\bullet(n), \de_0) $ by the number of children of the lowest non-root vertex:
\begin{equation}
\label{filtr-V-bul}
\bfzero = \mathcal{F}^1 V_\bullet(n) \subset
\mathcal{F}^2V_\bullet(n) \subset \mathcal{F}^3 V_\bullet(n) \subset \cdots \subset 
\mathcal{F}^n V_\bullet(n) = V_\bullet(n).
\end{equation}
Here $\mathcal{F}^p V_\bullet(n)$ is spanned by brace trees whose lowest non-root vertex has $\le p$ children. 
Then we consider the spectral sequence associated to this filtration.

The first differential, say $d_0$, splits vertices except for the lowest non-root vertex. Hence, 
\begin{equation}
\label{Gr-V-bul}
\Gr V_\bullet(n) \cong  \big( \bs \La \coAs_{\c} \odot \Br \big)(n)\,,
\end{equation}
where $\bs \La \coAs_{\c}$ is the collection with 
\begin{equation}
\label{bsi-LaAs-c}
\bs \La\coAs_{\c}(q) = \begin{cases}
  \bs^{2-q} \bbK[S_q] \otimes \sgn_q  \qquad {\rm if} ~~ q \ge 2 \,, \\
  \bfzero \qquad {\rm otherwise}\,.
\end{cases}
\end{equation}
Therefore, by inductive hypothesis, we conclude that
\begin{equation}
\label{E1-V-bul}
E_1 V_\bullet(n)  := H^{\bul}(\Gr V_\bullet(n)  , d_0) \cong  \big( \bs \La \coAs_{\c} \odot \Ger \big)(n).
\end{equation}
Moreover, the cohomology class in $H^{\bul}(\Gr^q V_\bullet(n)  , d_0)$ corresponding to the vector 
$$
\bs^{2-q} \ \id_q \otimes (v_1\otimes \dots \otimes v_q) \in 
\bs \La\coAs_{\c}(q) \otimes \big( \Ger(n_1) \otimes \dots \otimes \Ger(n_q) \big) 
$$
is represented by the $d_0$-cocycle
\begin{equation}
\label{d-0-cocycle}
\mu \big( T^{\bul}_q ; \Psi(v_1), \Psi(v_2), \dots, \Psi(v_q) \big) \in \Br(n),
\end{equation}
where $\mu$ is the operadic multiplication on $\Br$, $n = n_1 + \dots + n_q$, 
$T^{\bul}_q$ is the brace tree shown in figure \ref{fig:T-bul-q}, and 
$\Psi$ is the map of collections \eqref{Psi}. 
%
%
\begin{figure}[htp] 
\centering 
\begin{minipage}[t]{0.45\linewidth}
\centering
\begin{tikzpicture}[scale=0.5]
\tikzstyle{lab}=[circle, draw, minimum size=5, inner sep=1]
\tikzstyle{n}=[circle, draw, fill, minimum size=6, inner sep=0]
\tikzstyle{root}=[circle, draw, fill, minimum size=0, inner sep=1]
\node[root] (r) at (0, 0) {};
\node[n] (n) at (0, 1) {};
\node [lab] (v1) at (-2,2.8) {$\scriptstyle 1$};
\node [lab] (v2) at (-1,2.8) {$\scriptstyle 2$};
\draw (0.5,2.8) node[anchor=center] {{$\dots$}};
\node [lab] (vq) at (2,2.8) {$\scriptstyle q$};
\draw (r) edge (n);
\draw (n) edge (v1) edge (v2) edge (vq);
\end{tikzpicture}
\caption{The brace tree $T^{\bul}_{q}$} \label{fig:T-bul-q}
\end{minipage}
~~~
\begin{minipage}[t]{0.45\linewidth}
\centering
\begin{tikzpicture}[scale=0.5]
\tikzstyle{lab}=[circle, draw, minimum size=5, inner sep=1]
\tikzstyle{n}=[circle, draw, fill, minimum size=6, inner sep=0]
\tikzstyle{root}=[circle, draw, fill, minimum size=0, inner sep=1]
\node[root] (r) at (0, 0) {};
\node[n] (n) at (0, 1) {};
\node [lab] (v1) at (-3,2.8) {$\scriptstyle 1$};
\draw (-2,2.8) node[anchor=center] {{$\dots$}};
\node [lab] (v1i) at (-0.5,2.8) {$\scriptstyle i-1$};
\node [lab] (vi) at (1,2.8) {$\scriptstyle i$};
\node [lab] (vi1) at (1.2,4.3) {$\scriptstyle i+1$};
\node [lab] (vi2) at (2.5,2.8) {$\scriptstyle i+2$};
\draw (4,2.8) node[anchor=center] {{$\dots$}};
\node [lab] (vq) at (5,2.8) {$\scriptstyle q$};
\draw (r) edge (n);
\draw (n) edge (v1) edge (v1i) edge (vi) edge (vi2) edge (vq) (vi) edge (vi1);
\end{tikzpicture}
\caption{The brace tree $T^{\bul}_{q,i}$} \label{fig:T-bul-q-i}
\end{minipage}
\end{figure} 

Before proceeding to further pages of this spectral sequence, we need to fix some 
conventions\footnote{We use the cohomological version of the notational conventions from \cite[Construction 5.4.6]{Weibel}.}. 
First, we denote by $\cA^q_r$  ($r \ge 0$) the following subspaces of $\cF^q V_\bullet(n)$: 
\begin{equation}
\label{cA-q-r}
\cA^q_r : = \big\{ v \in \cF^q V_\bullet(n) ~\big|~ \de_0(v) \in \cF^{q-r} V_\bullet(n) \big\}.
\end{equation}
For example, $\cA^q_0 = \cF^q V_\bullet(n)$ and vectors in $\cA^q_1$ represent 
cocycles in $\Gr^q  V_\bullet(n)$. By the construction of the spectral sequence  \cite[Construction 5.4.6]{Weibel}, 
the components of the $r$-th page are the quotients 
\begin{equation}
\label{E-r-q}
E^q_r : = \frac{\cA^q_r }{\de_0( \cA^{q+r-1}_{r-1}) + \cA^{q-1}_{r-1}}\,.
\end{equation}

The results of the computation of $E_2 V_\bullet(n) : = H^{\bul} \big( E_1 V_\bullet(n), d_1 \big)$
are listed in the following claim:
\begin{claim}
\label{cl:E2-V-bullet}
For  $E_2 V_\bullet(n) : = H^{\bul} \big( E_1 V_\bullet(n), d_1 \big)$, we have
\begin{equation}
\label{E2-V-bul}
E_2 V_\bullet(n) \cong   \Com \odot \La \Lie (n) \big/ \La \Lie (n)  ~\oplus ~
\bs \big( \La\Com \odot \La \Lie (n) \big/ \La \Lie (n) \big)\,.
\end{equation}
More precisely, 
\begin{equation}
\label{E2-with-q}
E^q_2 V_\bullet(n) \cong
\begin{cases}
\displaystyle  \Com \odot \La \Lie (n) \big/ \La \Lie (n) ~\oplus~ 
\bigoplus_{n_1 + n_2 = n}  \Ind^{S_n}_{S_{n_1} \times S_{n_2}}\,
\Big( \sgn_{2} \otimes_{S_2} \big( \La \Lie (n_1) \otimes \La \Lie (n_2) \big) \Big) \quad 
{\rm if} ~~ q = 2,  \\[0.5cm]
\displaystyle  \bigoplus_{n_1 + \dots + n_q = n} 
\Ind^{S_n}_{S_{n_1} \times \dots \times S_{n_q}}
\Big( \bs^{2-q}\, \sgn_{q} \otimes_{S_q} \big(\La \Lie (n_1) \otimes \dots 
\otimes \La \Lie (n_q)\big) \Big) \quad {\rm if}~~ 3 \le q \le n, \\[0.5cm]
\qquad  \bfzero  \qquad {\rm otherwise}.
\end{cases}
\end{equation}
The classes corresponding to vectors in $\Com \odot \La \Lie (n) \big/ \La \Lie (n)$ are represented in 
$\cA^2_2$ by cocycles (in $(\Br(n), \de)$) which are obtained by applying $\Psi$ \eqref{Psi} to linear combinations 
of monomials \eqref{Ger-n-basis} in $\Ger(n)$ with $t \ge 2$.  

If $q \ge 3$, the class corresponding to the vector 
$$
\bs^{2-q}\, 1_q \otimes (v_1 \otimes \dots \otimes v_q) \in  \bs \La\Com \odot \La \Lie (n)
$$
is represented in $\cA^q_2$ by the cochain 
\begin{equation}
\label{fork-q-antisymm}
u_q : = \frac{1}{q!} \sum_{\si\in S_q} (-1)^{|\si|} \mu \big( \si ( T^{\bul}_q ) 
\otimes \mj(v_{1}) \otimes \mj(v_2) \otimes \dots \otimes \mj (v_q) \big)
\end{equation}
$$
~ + ~   
 \frac{1}{q!} \sum_{i=1}^{q-1}\sum_{\substack{\si\in S_q \\ \si(i) < \si(i+1)}} (-1)^{|\si|} 
\mu \big( \si ( T^{\bul}_{q,i} ) \otimes \mj(v_{1}) \otimes \mj(v_2) \otimes \dots \otimes \mj (v_q) \big),
$$
where $\mu$ is the operadic composition
$$
\mu : \Br(q) \otimes \Br(m_1)\otimes \dots \otimes \Br(m_q) \to \Br(m_1 + \dots + m_q),
$$ 
$\mj$ is the operad map in \eqref{mj}, and $T^{\bul}_q$ (resp. $T^{\bul}_{q,i}$) is the brace tree shown in figure \ref{fig:T-bul-q}
(resp. figure \ref{fig:T-bul-q-i}). Finally, the class corresponding to the vector 
$$
1_2 \otimes (v_1\otimes v_2) ~\in~  \sgn_{2} \otimes_{S_2} \big( \La \Lie (n_1) \otimes \La \Lie (n_2) \big)
$$
is represented in $\cA^2_2$ by the cochain 
\begin{equation}
\label{fork-q-2}
\frac{1}{2}\, \mu\big( (T_{\cup} - T_{\cup}^{\opp}) \otimes \mj(v_{1}) \otimes \mj(v_{2}) \big),
\end{equation}
where $T_{\cup}$ and $T_{\cup}^{\opp}$ are shown in figure \ref{fig:examples}.
\end{claim} 
\begin{remark}
\label{rem:not-a-cocycle}
Note that the vector \eqref{fork-q-antisymm} is not closed in $(V_\bullet(n), \de_0)$. It is merely 
a representative of an element in $E^q_2$, i.e. a vector $v \in \cF^q V_\bullet(n) $ such that 
$\de_0 (v) \in \cF^{q-2} V_\bullet(n)$. 
\end{remark}

\begin{proof}
The differential $d_1$ on $E_1 V_\bullet(n)$
splits the lowest non-root vertex producing a neutral child node with two children.
To describe this cochain complex, we consider the free
Gerstenhaber algebra $\Ger_n$ in $n$ auxiliary variables $a_1, a_2, \dots, a_n$ of degree 
zero. Forgetting the bracket $\{~,~\}$ on $\Ger_n$ we can view it merely as the 
free commutative algebra (without unit)
$$
\Ger_n = \Com (\La\Lie_n)
$$
generated by the free $\La\Lie$-algebra $\La\Lie_n$ in the auxiliary variables $a_1, a_2, \dots, a_n$ 

Next, we introduce the cofree coassociative coalgebra 
\begin{equation}
\label{coAs-bsi-Ger-n}
\coAs(\bsi \Ger_n) = \bigoplus_{q \ge 1} \big( \bsi \Ger_n \big)^{\otimes\, q}
\end{equation}
and equip it with the coderivation  $\md$ defined by the 
equation\footnote{Note that, since the coalgebra \eqref{coAs-bsi-Ger-n} is 
cofree, any coderivation $\md$ is uniquely determined by its composition
with the projection \eqref{p-Ger-n}.}
\begin{equation}
\label{coder-md}
p \circ \md (\bsi v_1 \otimes \dots \otimes \bsi v_q) = 
\begin{cases}
 (-1)^{|v_1|+1} \bsi v_1 v_2  \qquad {\rm if} ~~  q = 2\,, \\
 0  \qquad {\rm otherwise}\,,
\end{cases}
\end{equation}
where $v_i \in \Ger_n$ and $p$ is the canonical projection;
\begin{equation}
\label{p-Ger-n}
p : \coAs(\bsi \Ger_n) \to  \bsi \Ger_n\,.
\end{equation}

It is easy to see that the coderivation $\md$ has degree $1$.
Moreover, due to associativity of the multiplication on $\Ger_n$, 
we have 
$$
\md^2 = 0\,. 
$$
In other words, $\md$ is a differential on the coalgebra 
\eqref{coAs-bsi-Ger-n}.

For our purposes we need the following truncation of the 
cochain complex $\bs^2 \coAs(\bsi \Ger_n) $ 
\begin{equation}
\label{trunc-T-Ger-n}
\bs^2 T' (\bsi \Ger_n) =   \bigoplus_{q \ge 2} \bs^2 \big( \bsi \Ger_n \big)^{\otimes\, q}
\end{equation}
with the differential $\md'$ given by the formula:
\begin{equation}
\label{md-pr}
\md' \big( \bs^2  ( \bsi v_1 \otimes \bsi v_2 \otimes \dots \otimes \bsi v_q) \big) = 
\begin{cases}
\bs^2 \md ( \bsi v_1 \otimes \bsi v_2 \otimes \dots \otimes \bsi v_q)  \qquad {\rm if} ~~ q > 2  \,, \\
 0 \qquad {\rm if} ~~ q = 2\,,
\end{cases}
\qquad v_i \in \Ger_n\,.
\end{equation}

It is not hard to see that $E_1 V_\bullet(n)$  \eqref{E1-V-bul} is isomorphic
to the subspace of $\bs^2 T' (\bsi \Ger_n)$ which is spanned by 
tensor monomials
$$
\bs^2 ( \bsi v_1 \otimes \bsi v_2 \otimes \dots \otimes \bsi v_q)\,, 
\qquad v_i \in \Ger_n\,, \qquad 2 \le q \le n
$$
in which each variable from the set $\{a_1, a_2, \dots, a_n\}$
appears exactly once.  It is easy to see that this subspace is 
a subcomplex with respect to $\md'$ and, moreover, the differential 
$d_1$ coincides with the restriction of $\md'$ up to a total sign. 

Since the augmentation 
\begin{equation}
\label{aug-coAs-Ger-n}
\dots ~\stackrel{\md}{\longrightarrow}~
\big( \bsi \Ger_n \big)^{\otimes\, 2} ~\stackrel{\md}{\longrightarrow}~
 \bsi \Ger_n  ~\stackrel{0}{\longrightarrow}~ \bbK
\end{equation}
of the cochain complex \eqref{coAs-bsi-Ger-n} computes the Hochschild homology 
\begin{equation}
\label{HH}
HH_{-\bul} (S(\La\Lie_n), \bbK)
\end{equation}
of the free commutative algebra $S(\La\Lie_n)$ (with unit) with the trivial 
coefficients, we conclude that\footnote{In \cite{Loday}, J.-L. Loday only considers the 
case when the symmetric algebra is generated by an ``ungraded'' vector space and 
$HH_{\bul}$ is computed with coefficients in the symmetric algebra. However, the obvious 
generalization to the Koszul resolution to the graded case can be applied in the 
straightforward manner in our case.} 
\cite[Section 3.2]{Loday}
\begin{equation}
\label{H-coAs-Ger-n}
H^{\bul} (\coAs(\bsi \Ger_n), \md) = \bigoplus_{q \ge 1} S^q (\bsi \La\Lie_n)\,, 
\end{equation}
and the cohomology class of the symmetric word
$(\bsi v_1, \bsi v_2, \dots, \bsi v_q) \in  S^q (\bsi \La\Lie_n)$ is represented 
by the cocycle: 
$$
\frac{1}{q!}\, \sum_{\si \in S_q} (-1)^{\ve(\si, v_1, \dots, v_q)} (\bsi v_{\si(1)}, \bsi v_{\si(2)}, \dots, \bsi v_{\si(q)})
~\in~ \big( \bsi \Ger_n \big)^{\otimes\, q}\,,
$$
where the sign factors $ (-1)^{\ve(\si, v_1, \dots, v_q)}$ are determined by the Koszul rule. 

When we pass to the truncation \eqref{trunc-T-Ger-n} of the Hochschild complex, the cohomology in 
the terms $\big( \bsi \Ger_n \big)^{\otimes\, q}$ for $q \ge 3$ does not change. 

As for $q = 2$, all vectors in 
$$
\big( \bsi \Ger_n \big)^{\otimes\, 2}
$$
are cocycles in the truncated complex \eqref{trunc-T-Ger-n}. 

Since for every pair of vectors $v_1, v_2 \in \Ger_n$  
$$
\bsi v_1 \otimes  \bsi v_2 = 
$$
$$
\frac{(-1)^{|v_1|}}{2} \bsi \otimes \bsi 
(v_1 \otimes  v_2  + (-1)^{|v_1| |v_2|}  v_2 \otimes  v_1)
\, + \, \frac{1}{2} (\bsi v_1 \otimes  \bsi v_2  + (-1)^{(|v_1| +1) (|v_2| +1)} \bsi v_2 \otimes  \bsi v_1),
$$
~\\
we have the obvious decomposition
$$
\big( \bsi \Ger_n \big)^{\otimes\, 2}   \cong  \bs^{-2}\, S^{\ge 2} (\La\Lie_n)  ~\oplus~ S^2 \big( \bsi \Ger_n \big), 
$$
where $S^2 \big( \bsi \Ger_n \big)$ is precisely the kernel of 
\begin{equation}
\label{diff-needed}
\big( \bsi \Ger_n \big)^{\otimes\, 2} ~\stackrel{\md}{\longrightarrow}~ \bsi \Ger_n
\end{equation}
and $\bs^{-2} S^{\ge 2} (\La\Lie_n)$ is (up to the degree shift) the image of \eqref{diff-needed}.

Combining this observation with the knowledge about homology \eqref{HH}, we conclude that 
\begin{equation}
\label{H-trunc-T-Ger-n}
H^{\bul}\big( \bs^2 T' (\bsi \Ger_n), \md' \big) \cong  S^{\ge 2}(\La\Lie_n)  ~\oplus~
  \bigoplus_{q \ge 2} \bs^2 S^q (\bsi \La\Lie_n).
\end{equation}

On the other hand,  $E_1 V_\bullet(n)$ is isomorphic to the direct 
summand of the cochain complex $\big( \bs^2 T' (\bsi \Ger_n), \md' \big) $\,. 

Thus the first two statements of Claim \ref{cl:E2-V-bullet} follow from \eqref{H-trunc-T-Ger-n}. 
To deduce the remaining statements, we use the description of cohomology classes in $H^{\bul}(\Gr V_\bullet(n)  , d_0)$
corresponding to vectors in $\big( \bs \La \coAs_{\c} \odot \Ger \big)(n)$ 
(see eq. \eqref{d-0-cocycle}). 

The most involving statement is about the class corresponding to the vector    
\begin{equation}
\label{for-q-ge-3}
\bs^{2-q}\, 1_q \otimes (v_1 \otimes \dots \otimes v_q) \in  \bs \La\Com \odot \La \Lie (n)
\end{equation}
for $q \ge 3$. 

Using the information about the $E_1$ page, we know that \eqref{for-q-ge-3} is 
represented in $\cA^q_1$ by the vector 
\begin{equation}
\label{f-q}
f_q : =
\frac{1}{q!} \sum_{\si\in S_q} (-1)^{|\si|} \mu \big( \si ( T^{\bul}_q ) 
\otimes \mj(v_{1}) \otimes \mj(v_2) \otimes \dots \otimes \mj (v_q) \big).
\end{equation}
A direct computation shows that
$$
\de_0 \Big( \sum_{\si\in S_q} (-1)^{|\si|} \si  ( T^{\bul}_q ) \Big)  +
\de_0 \Big( \sum_{i=1}^{q-1}\sum_{\substack{\si\in S_q \\ \si(i) < \si(i+1)}} (-1)^{|\si|} 
 \si ( T^{\bul}_{q,i} )  \Big) \in \cF^{q-2} V_\bullet(n). 
$$
Therefore the sum
$$
u_q = f_q + 
 \frac{1}{q!} \sum_{i=1}^{q-1}\sum_{\substack{\si\in S_q \\ \si(i) < \si(i+1)}} (-1)^{|\si|} 
\mu \big( \si ( T^{\bul}_{q,i} ) \otimes \mj(v_{1}) \otimes \mj(v_2) \otimes \dots \otimes \mj (v_q) \big)
$$
belongs to $\cA^q_2$ and represents the element in $E^q_2$ corresponding to \eqref{for-q-ge-3}. 
 
Claim \ref{cl:E2-V-bullet} is proved.
\end{proof}

Due to Lemma \ref{lem:specseqabutment} from Appendix \ref{app:E2-Einfty},
this spectral sequence degenerates at the second page, i.e.,
\begin{equation}
\label{abuts-at-E2}
E_{\infty} V_\bullet(n) = E_2 V_\bullet(n).
\end{equation}

Hence Claim \ref{cl:E2-V-bullet} implies the following statement.

%
%
\begin{claim}
\label{cl:V-bullet}
For the complex $(V_{\bullet}(n), \de_0)$ we have  
\begin{equation}
\label{H-V-bullet}
H^{\bul}(V_\bullet(n), \delta_0)\cong   \Com \odot \La \Lie (n) \big/ \La \Lie (n)  ~\oplus ~
\bs \big( \La\Com \odot \La \Lie (n) \big/ \La \Lie (n) \big).
\end{equation}
Cohomology classes in $(V_{\bullet}(n), \de_0)$
corresponding to vectors in $\Com \odot \La \Lie (n) \big/ \La \Lie (n)$ are represented 
by cocycles (in $(\Br(n), \de)$) which are obtained by applying $\Psi$ \eqref{Psi} to linear combinations 
of monomials \eqref{Ger-n-basis} in $\Ger(n)$ with $t \ge 2$.  
The class corresponding to the vector 
$$
\bs^{2-q}\, 1_q \otimes (v_1 \otimes \dots \otimes v_q) \in  \bs \La\Com \odot \La \Lie (n), \qquad q \ge 2
$$
is represented in $(V_{\bullet}(n), \de_0)$ by the $\de_{0}$-cocycle of the form 
\begin{equation}
\label{cocycle-with-uq}
u_q + \dots
\end{equation}
where $u_q$ is the vector given in \eqref{fork-q-antisymm} and 
$\dots$ denotes the sum of terms in $\cF^{q-1} V_{\bullet}(n)$.
\end{claim}
\begin{remark}
\label{rem:claimdual}
One may, of course, dualize the statement of Claim \ref{cl:V-bullet}. The dual statement says that
\begin{equation}
\label{H-V-bul-decomp}
H^{\bul}(V_\bullet^*(n), \delta_0^*) \cong X^* \oplus U^*\,,
\end{equation}
where $X^* \subset \Ger(n)^*$ is the kernel of $\Ger(n)^*\to \La\Lie(n)^*$ and 
$U^*$ is the linear dual of 
$$
\bs \big( \La\Com \odot \La \Lie (n) \big/ \La \Lie (n) \big). 
$$
\end{remark}

\subsection{A technical claim about $H^{\bul}(\de^*_1) : H^{\bul}(V_{\bul}(n)^*, \de^*_0) \to H^{\bul}(V_{\circ}(n)^*, \de^*_0)$}
\label{sec:H-de-1}

Let summarize what we proved so far: 
\begin{itemize}

\item First, due to Claim \ref{cl:2}, 
\begin{equation}
\label{H-V-circ}
H^{k}(V_{\circ}(n), \delta_0)  =
\begin{cases}
\bbK[S_n] \qquad {\rm if} ~~ k =1-n  \,, \\
 0  \qquad {\rm otherwise}\,.
\end{cases}
\end{equation}

\item Second, due to Claim \ref{cl:V-bullet},
$$
H^{\bul}(V_\bullet(n), \delta_0) ~\cong ~  \Com \odot \La \Lie (n) \big/ \La \Lie (n)  ~\oplus ~
\bs \big( \La\Com \odot \La \Lie (n) \big/ \La \Lie (n) \big).
$$

\item The subspace 
\begin{equation}
\label{U}
\bs \big( \La\Com \odot \La \Lie (n) \big/ \La \Lie (n) \big)
\end{equation}
is concentrated in the degree $2-n$, and the subspace 
$$
\Com \odot \La \Lie (n) \big/ \La \Lie (n) 
$$ 
lives in degrees $2-n \le \tiny{\bullet} \le 0$.

\end{itemize}

Thus the operator $H^{\bul} (\de_1)$ sends vectors of  $H^{1-n}(V_{\circ}(n), \delta_0) $ to 
the space $H^{2-n}(V_\bullet(n), \delta_0)$. Hence, 
$$
H^{k} (\Br(n)) \cong 
\begin{cases}
H^{1-n}(V_{\circ}(n), \delta_0) \cap \ker \big( H^{\bul} (\de_1) \big) \quad {\rm if} ~~~ k=1-n  \,, \\[0.3cm]
H^{2-n}(V_\bullet(n), \delta_0) \big/ \mathrm{Im} \big( H^{\bul} (\de_1) \big) \quad {\rm if} ~~~ k=2-n  \,, \\[0.3cm]
H^{k}(V_\bullet(n), \delta_0) \qquad {\rm if} ~~~ 3-n \le k \le 0  \,, \\[0.3cm]
0  \qquad {\rm otherwise}\,.
\end{cases}
$$ 

Let us prove that
\begin{claim}
\label{cl:1-n-lower-bound}
The map 
$$
H^{\bul}(\mj) : \La\Lie(n) \to  H^{1-n} (\Br(n))
$$
is injective. In particular,  
\begin{equation}
\label{dim-for-1-n}
\dim\,  H^{1-n} (\Br(n)) \ge (n-1)!
\end{equation}
\end{claim}
\begin{proof}
Since $\Br(n)$ lives is degrees $1-n \le \bullet \le 0$, 
\begin{equation}
\label{H-1n}
H^{1-n} (\Br(n)) = \Br(n)^{1-n} \cap \ker(\de).
\end{equation}

It is not hard to prove (by induction on $n$) that 
$$
\begin{tikzpicture}
[lab/.style={draw, circle, minimum size=5, inner sep=1.5}, 
n/.style={draw, circle, fill, minimum size=5, inner sep=1}, scale=.7]
\draw  (-5,2) node { $\mj\big( \leftbr a_1, a_2\}, a_3 \} \dots, a_n\} \big) ~ = ~$};
\draw  (-1,2) node { $\pm$};
\draw (0,0) node[lab] (v1) {$\scriptstyle 1$};
\draw (0,-0.7) node[root] (r) {};
\draw (0,1) node[lab] (v2) {$\scriptstyle 2$};
\draw  (0,2.5) node { $\vdots$};
\draw (0,4) node[lab] (vn) {$\scriptstyle n$};
\draw (v1) edge (r) edge (v2) (v2) edge (0,1.5) (vn) edge (0,3.5);
\draw  (2,2) node { $+ ~  \dots $};
\end{tikzpicture}
$$
where $\dots$ is the sum of braces trees which do not involve string-like brace trees 
with vertex $1$ at the lowest position. 

Therefore, for every permutation $\tau \in S_{\{2, 3, \dots, n\}}$, we have  
$$
\begin{tikzpicture}
[lab/.style={draw, circle, minimum size=5, inner sep=1}, 
n/.style={draw, circle, fill, minimum size=5, inner sep=1}, scale=.7]
\draw  (-7,2) node { $\mj\big( \leftbr a_1, a_{\tau(2)}\}, a_{\tau(3)} \} \dots, a_{\tau(n)}\} \big) ~ = ~$};
\draw  (-1.5,2) node { $\pm$};
\draw (0,-0.2) node[lab] (v1) {$\scriptstyle 1$};
\draw (0,-1) node[root] (r) {};
\draw (0,1) node[lab] (v2) {$\scriptstyle \tau(2)$};
\draw  (0,2.5) node { $\vdots$};
\draw (0,4) node[lab] (vn) {$\scriptstyle \tau(n)$};
\draw (v1) edge (r) edge (v2) (v2) edge (0,2) (vn) edge (0,3);
\draw  (2.5,2) node { $+ ~~  \dots $};
\end{tikzpicture}
$$
where, as above, $\dots$ is the sum of braces trees which do not involve string-like brace trees 
with vertex $1$ at the lowest position. 

Thus, $\mj$ gives us $(n-1)!$ linearly independent vectors
$$
\big\{\, \mj\big( \leftbr a_1, a_{\tau(2)}\}, a_{\tau(3)} \} \dots, a_{\tau(n)}\} \big) \, \big\}_{\tau \in S_{\{2, 3, \dots, n\}} }
$$
in \eqref{H-1n}. 
 
Since the set 
$$
\big\{\, \leftbr a_1, a_{\tau(2)}\}, a_{\tau(3)} \} \dots, a_{\tau(n)}\}  \,  \big\}_{\tau \in S_{\{2, 3, \dots, n\}} }
$$
is a basis of $\La\Lie(n)$, the claim follows. 
\end{proof}

To prove the other inequality 
\begin{equation}
\label{dim-less-or-eq}
\dim\,  H^{1-n} (\Br(n)) \le (n-1)!
\end{equation}
we need the following technical statement: 
%
%
\begin{claim}
\label{cl:shuffles}
Let $1\le r \le n-1$ and
$$
\si = 
\left(
\begin{array}{cccccccc}
 1 & 2 & \dots  &  r & r+1& r+2  & \dots & n   \\
 i_1 & i_2 & \dots  &  i_r & j_1& j_2 &  \dots & j_{n-r}   \\
\end{array}
\right)
$$
be a permutation in $S_n$. Let $T_{||, \si, r}$ (resp.  $T^{\opp}_{||, \si, r}$) be the 
brace tree shown in figure \ref{fig:T-llr} (resp. in figure \ref{fig:T-llr-opp}).

The vector 
\begin{equation}
\label{VIP}
\frac{1}{2} (T_{||, \si, r} + (-1)^{r(n-r)} T^{\opp}_{||, \si, r} )
\end{equation}
is a cocycle in the dual complex $(V_{\bul}(n)^*, \de^*_0)$ representing 
a cohomology class corresponding to a vector in $U^*$, i.e. the dual of 
the subspace \eqref{U}. 

Moreover, the vector 
\begin{equation}
\label{de-1-VIP}
\frac{1}{2} \,\de^*_1 (T_{||, \si, r} + (-1)^{r(n-r)} T^{\opp}_{||, \si, r} )
\end{equation}
is cohomologous in $(V^*_{\circ}(n), \de^*_{0} )$ to  
\begin{equation}
\label{unshuffle-T-si}
\sum_{\tau \in \Sh_{r, n-r}}  (-1)^{|\tau|}\, T^n_{\si \circ \tau^{-1} }\,, 
\end{equation}
where $(-1)^{|\tau|}$ is the sign of the permutation $\tau$ and 
$\{ T^n_{\la} \}_{\la \in S_n}$ be the family of brace trees 
shown in figure \ref{fig:T-la}.
\end{claim}
\begin{figure}[htp] 
\centering 
\begin{minipage}[t]{0.45\linewidth}
\centering
\begin{tikzpicture}
[lab/.style={draw, circle, minimum size=5, inner sep=1}, 
n/.style={draw, circle, fill, minimum size=5, inner sep=1}, scale=.7]
\draw (0,-0.7) node[root] (r) {};
\draw (0,0) node[n] (n) {};
\draw (-0.7,0.8) node[lab] (vi1) {$\scriptstyle i_1$};
\draw (-0.7, 2) node[lab] (vi2) {$\scriptstyle i_2$};
\draw  (-0.7,3.5) node { $\vdots$};
\draw (-0.7, 5) node[draw, circle, minimum size=5, inner sep=3] (vir) {$\scriptstyle i_r$};
\draw (0.7,0.8) node[lab] (vj1) {$\scriptstyle j_1$};
\draw (0.7, 2) node[lab] (vj2) {$\scriptstyle j_2$};
\draw  (0.7,3.5) node { $\vdots$};
\draw (0.7, 5) node[lab] (vjrn) {$\scriptstyle j_{n-r}$};
\draw (n) edge (r) edge (vi1) edge (vj1)  (vi1) edge (vi2)  (vj1) edge (vj2) 
(vi2) edge (-0.7,2.6) (vj2) edge (0.7,2.6) (vir) edge (-0.7,4) (vjrn) edge (0.7, 4);
\end{tikzpicture}
\caption{The brace tree $T_{||,\si, r}$} \label{fig:T-llr}
\end{minipage}
~
\begin{minipage}[t]{0.45\linewidth}
\centering
\begin{tikzpicture}
[lab/.style={draw, circle, minimum size=5, inner sep=1}, 
n/.style={draw, circle, fill, minimum size=5, inner sep=1}, scale=.7]
\draw (0,-0.7) node[root] (r) {};
\draw (0,0) node[n] (n) {};
\draw (0.7,0.8) node[lab] (vi1) {$\scriptstyle i_1$};
\draw (0.7, 2) node[lab] (vi2) {$\scriptstyle i_2$};
\draw  (0.7,3.5) node { $\vdots$};
\draw (0.7, 5) node[draw, circle, minimum size=5, inner sep=3] (vir) {$\scriptstyle i_r$};
\draw (-0.7,0.8) node[lab] (vj1) {$\scriptstyle j_1$};
\draw (-0.7, 2) node[lab] (vj2) {$\scriptstyle j_2$};
\draw  (-0.7,3.5) node { $\vdots$};
\draw (-0.7, 5) node[lab] (vjrn) {$\scriptstyle j_{n-r}$};
\draw (n) edge (r) edge (vi1) edge (vj1)  (vi1) edge (vi2)  (vj1) edge (vj2) 
(vi2) edge (0.7,2.6) (vj2) edge (-0.7,2.6) (vir) edge (0.7,4) (vjrn) edge (-0.7, 4);
\end{tikzpicture}
\caption{The brace tree $T^{\opp}_{||,\si, r}$} \label{fig:T-llr-opp}
\end{minipage}
\end{figure}
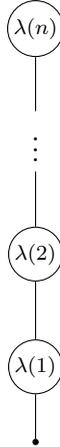
\begin{figure}[htp] 
\centering 
\begin{tikzpicture}
\node[root] (r) at (0,0) {};
\node[ext] (t1) at (0,1) {$\scriptstyle \la(1)$};
\node[ext] (t2) at (0,2.5) {$\scriptstyle \la(2)$};
\node[] (dots) at (0,4) {$\vdots$};
\node[ext] (tn) at (0,5.5) {$\scriptstyle \la(n)$};
\draw (r) edge (t1) (t1) edge (t2) (t2) edge (dots) (dots) edge (tn);
\end{tikzpicture}
\caption{The brace tree $T^n_{\la}$. Here $\la \in S_n$} \label{fig:T-la}
\end{figure}

\begin{proof}
First, every brace tree with exactly one neutral vertex (at the lowest position) is 
a cocycle in  $(V_{\bul}(n)^*, \de^*_0)$. 

To prove that the vector \eqref{VIP} belongs to $U^*$, we need to show that 
the pairing
\begin{equation}
\label{pairing-zero}
(T_{||, \si, r} + (-1)^{r(n-r)} T^{\opp}_{||, \si, r} )\, (w) = 0,
\end{equation}
where $w$ is a cocycle representing a cohomology class in $H^{\bul}(V_\bullet(n), \delta_0)$
corresponding a vector in $\Com \odot \La \Lie (n) \big/ \La \Lie (n)$. 

Due to Claim \ref{cl:V-bullet}, we may assume that 
$$
w = \Psi(c ),
$$
where $c$ is a linear combination of monomials \eqref{Ger-n-basis} in $\Ger(n)$ with $t  = 2$.  

Since $\Psi(c)$ is a linear combination of expressions of the form 
$$
\si \circ \mu \big( (T_{\cup} + T^{\opp}_{\cup})  \otimes \mj(h_1) \otimes \mj(h_2) \big),
$$ 
where $h_1 \in \La\Lie(n_1)$, $h_2 \in \La\Lie(n - n_1)$, $\mu$ is the operadic 
multiplication, and $\si \in S_{n}$, 
the vector $\Psi(c)$ is anti-symmetric with respect to the $S_2$ action 
on $\cF^2 V_\bullet(n)$ which switches the two branches originating from 
the lowest non-root vertex.   
 
On the other hand, the vector \eqref{VIP} is symmetric with respect to 
this $S_2$ action. Hence \eqref{pairing-zero} follows.   
 
We will now prove that 
\begin{equation}
\label{de-1-VIP-new}
\de^*_1 (T_{||, \si, r} )  ~ - ~ \sum_{\tau \in \Sh_{r, n-r}}  (-1)^{|\tau|}\, T^n_{\si \circ \tau^{-1} } ~~\in~~ \de^*_0 \big( V^*_{\circ}(n) \big).
\end{equation}
Then the desired statement about the vector \eqref{de-1-VIP} will follow from the graded commutativity of 
the shuffle product. 

The simple calculation shown in figure \ref{fig:case-n2} proves 
\eqref{de-1-VIP-new} in the case when $n=2$ (and $r=1$). 
%
%
\begin{figure}[htp] 
\centering
\begin{tikzpicture}
[lab/.style={draw, circle, minimum size=5, inner sep=1}, 
n/.style={draw, circle, fill, minimum size=5, inner sep=1}, scale=.7]
\draw  (-2,0) node { $\de^*_1 $};
\draw (0,-0.7) node[root] (rr) {};
\draw (0,0) node[n] (n) {};
\draw (-0.7, 0.7) node[lab] (v1) {$\scriptstyle 1$};
\draw (0.7, 0.7) node[lab] (v2) {$\scriptstyle 2$};
\draw (n) edge (rr) edge (v1) edge (v2); 
\draw  (2,0) node { $ = $};
\begin{scope}[shift={(4,0)}]
\draw (0,-0.7) node[root] (rr) {};
\draw (0, 0) node[lab] (v1) {$\scriptstyle 1$};
\draw (0, 0.9) node[lab] (v2) {$\scriptstyle 2$};
\draw (v1) edge (rr) edge (v2); 
\end{scope}
\draw  (6,0) node { $ - $};
\begin{scope}[shift={(8,0)}]
\draw (0,-0.7) node[root] (rr) {};
\draw (0, 0) node[lab] (v2) {$\scriptstyle 2$};
\draw (0, 0.9) node[lab] (v1) {$\scriptstyle 1$};
\draw (v2) edge (rr) edge (v1); 
\end{scope}
\end{tikzpicture}
\caption{The proof of \eqref{de-1-VIP-new} in the case $n=2$} \label{fig:case-n2}
\end{figure}
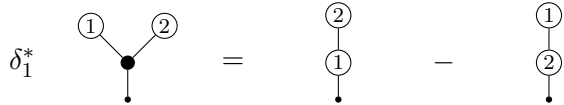
This also settles the base of our induction. 

Next, we observe that the linear combination 
$$
\begin{tikzpicture}
[lab/.style={draw, circle, minimum size=5, inner sep=1}, 
n/.style={draw, circle, fill, minimum size=5, inner sep=1}, scale=.7]
\draw  (-4.5,1) node { $\de^*_1(T_{||, \si, r} ) ~~  $};
\draw  (-2.5,1) node { $ - \quad \de^*_0 $};
\draw (0,-0.7) node[root] (r) {};
\draw (0,0) node[lab] (vi1) {$\scriptstyle i_1$};
\draw (0,1) node[n] (n) {};
\draw (-0.7, 1.8) node[lab] (vi2) {$\scriptstyle i_2$};
\draw  (-0.7,3.5) node { $\vdots$};
\draw (-0.7, 5) node[draw, circle, minimum size=5, inner sep=3] (vir) {$\scriptstyle i_r$};
\draw (0.7,1.8) node[lab] (vj1) {$\scriptstyle j_1$};
\draw (0.7, 3) node[lab] (vj2) {$\scriptstyle j_2$};
\draw  (0.7,4.5) node { $\vdots$};
\draw (0.7, 6) node[lab] (vjrn) {$\scriptstyle j_{n-r}$};
\draw (r) edge (vi1)  (n) edge (vi1) edge (vi2) edge (vj1)
(vj1) edge (vj2) 
(vi2) edge (-0.7,2.6) (vj2) edge (0.7,3.6) (vir) edge (-0.7,4) (vjrn) edge (0.7, 5);
\draw  (3,1) node { $ - \quad (-1)^r ~ \de_0^* $};
\begin{scope}[shift={(6,0)}]
\draw (0,-0.7) node[root] (r) {};
\draw (0,0) node[lab] (vj1) {$\scriptstyle j_1$};
\draw (0,1) node[n] (n) {};

\draw (-0.7,1.8) node[lab] (vi1) {$\scriptstyle i_1$};
\draw (-0.7, 3) node[lab] (vi2) {$\scriptstyle i_2$};
\draw  (-0.7,4.5) node { $\vdots$};
\draw (-0.7, 6) node[draw, circle, minimum size=5, inner sep=3] (vir) {$\scriptstyle i_r$};
\draw (0.7, 1.8) node[lab] (vj2) {$\scriptstyle j_2$};
\draw  (0.7,3.5) node { $\vdots$};
\draw (0.7, 5) node[lab] (vjrn) {$\scriptstyle j_{n-r}$};
\draw  (r) edge (vj1) (n) edge (vj1) edge (vi1) edge (vj2) 
(vi1) edge (vi2)  (vi2) edge (-0.7,3.6) (vj2) edge (0.7,2.6) (vir) edge (-0.7,5) (vjrn) edge (0.7, 4);
\end{scope}
\end{tikzpicture}
$$
is obtained from 
$$
\begin{tikzpicture}
[lab/.style={draw, circle, minimum size=5, inner sep=1}, 
n/.style={draw, circle, fill, minimum size=5, inner sep=1}, scale=.7]
\draw  (-2.5,1) node { $ - \quad \de^*_0 $};
\draw (0,-0.7) node[root] (r) {};
\draw (0,0) node[lab] (vi1) {$\scriptstyle i_1$};
\draw (0,1) node[n] (n) {};
\draw (-0.7, 1.8) node[lab] (vi2) {$\scriptstyle i_2$};
\draw  (-0.7,3.5) node { $\vdots$};
\draw (-0.7, 5) node[draw, circle, minimum size=5, inner sep=3] (vir) {$\scriptstyle i_r$};
\draw (0.7,1.8) node[lab] (vj1) {$\scriptstyle j_1$};
\draw (0.7, 3) node[lab] (vj2) {$\scriptstyle j_2$};
\draw  (0.7,4.5) node { $\vdots$};
\draw (0.7, 6) node[lab] (vjrn) {$\scriptstyle j_{n-r}$};
\draw (r) edge (vi1)  (n) edge (vi1) edge (vi2) edge (vj1)
(vj1) edge (vj2) 
(vi2) edge (-0.7,2.6) (vj2) edge (0.7,3.6) (vir) edge (-0.7,4) (vjrn) edge (0.7, 5);
\draw  (3,1) node { $ - \quad (-1)^r ~ \de_0^* $};
\begin{scope}[shift={(6,0)}]
\draw (0,-0.7) node[root] (r) {};
\draw (0,0) node[lab] (vj1) {$\scriptstyle j_1$};
\draw (0,1) node[n] (n) {};

\draw (-0.7,1.8) node[lab] (vi1) {$\scriptstyle i_1$};
\draw (-0.7, 3) node[lab] (vi2) {$\scriptstyle i_2$};
\draw  (-0.7,4.5) node { $\vdots$};
\draw (-0.7, 6) node[draw, circle, minimum size=5, inner sep=3] (vir) {$\scriptstyle i_r$};
\draw (0.7, 1.8) node[lab] (vj2) {$\scriptstyle j_2$};
\draw  (0.7,3.5) node { $\vdots$};
\draw (0.7, 5) node[lab] (vjrn) {$\scriptstyle j_{n-r}$};
\draw  (r) edge (vj1) (n) edge (vj1) edge (vi1) edge (vj2) 
(vi1) edge (vi2)  (vi2) edge (-0.7,3.6) (vj2) edge (0.7,2.6) (vir) edge (-0.7,5) (vjrn) edge (0.7, 4);
\end{scope}
\end{tikzpicture}
$$
by retaining only the terms which are obtained by contracting only the edges which 
are adjacent to the neutral vertex and lie above this neutral vertex.  

Thus the inductive step follows from the fact that 
the set of shuffles $\Sh_{r, n-r}$ splits into the disjoint union 
of permutations of the form 
$$
\left(
\begin{array}{ccccccc}
 1 &   2  &   \dots  &   r      &  r+1      & \dots &   n \\
 1 & \si(2)  & \dots & \si(r ) & \si(r+1) & \dots & \si(n)    
\end{array}
\right) 
$$
with $\si \in S_{\{2,3, \dots, n\}}$, $\si(2) < \si(3) < \dots < \si(r )$,  $\si(r+1) < \si(r+2) < \dots < \si(n )$, and 
permutations of the form 
$$
\left(
\begin{array}{cccccccc}
 1 &   2  &   \dots  &   r      &  r+1     & r+2 & \dots &   n \\
 \si(1) & \si(2)  & \dots & \si(r ) & 1 & \si(r+2) &\dots & \si(n)    
\end{array}
\right), 
$$
where $\si$ is a bijection $\si: \{1,2,\dots, r, r+2, \dots, n\}$ to 
$\{2,3, \dots, n\}$ such that  $\si(1) < \si(2) < \dots < \si(r )$ and 
 $\si(r+2) < \si(r+3) < \dots < \si(n)$. 
 
Claim \ref{cl:shuffles} is proved.  
\end{proof}

Let us recall that, 
for every graded vector space\footnote{The isomorphism \eqref{coLie-V} is the dual version of  \cite[Proposition 1.3.5]{Loday-Vallette}.} $V$,
\begin{equation}
\label{coLie-V}
\coLie(V) \cong \coAs(V) /  \coAs(V) \bullet_{\Sh} \coAs(V),
\end{equation}
where $\bullet_{\Sh}$ denotes the shuffles product. 

Thus Claims \ref{cl:2} and \ref{cl:shuffles} imply that 
$$
\dim\,  H^{n-1} (\Br(n)^*) \le (n-1)!
$$
and the desired inequality \eqref{dim-less-or-eq} follows.

\subsection{The final strokes} 
\label{sec:the-end}

Combining Claim \ref{cl:1-n-lower-bound} with the inequality \eqref{dim-less-or-eq}, we conclude that 
\begin{equation}
\label{dim-for-1-n-done}
\dim\,  H^{1-n} (\Br(n)) = (n-1)!
\end{equation}
and the restriction $\displaystyle \Psi \big|_{\La\Lie(n)}$ induces an 
isomorphism 
$$ 
\La\Lie(n) \cong H^{1-n} (\Br(n)).
$$

Hence, due to the summary given on page \pageref{H-V-circ} and the 
second statement of Claim \ref{cl:V-bullet}, it suffices to show that
\begin{equation}
\label{desired}
\coker \Big(   H^{\bul}(V_\circ(n), \delta_0)~  \xrightarrow{ ~~~~H^{\bul}(\de_1) \hspace*{0.2cm}}   ~  H^{\bul}(V_{\bul}(n), \delta_0) \Big) ~\cong~
 \Com \odot \La \Lie (n) \big/ \La \Lie (n).
\end{equation}

The later is a consequence of \eqref{H-V-bullet}, the equality 
$\dim\,  H^{n-1} (\Br(n)^*) = (n-1)!$ and Claim \ref{cl:shuffles}.
Indeed, due to  Claim \ref{cl:shuffles} and equality $\dim\,  H^{n-1} (\Br(n)^*) = (n-1)!$, 
the dimension of the space
$$
H^{\bul}(\de_1^*) \big(U^*\big)
$$
should be equal to $n! - (n-1)!$, where $U^*$ is the linear dual of \eqref{U}.

On the other hand, $\dim (U) = n! - (n-1)! = \dim (U^*)$ and hence the restriction of 
$H^{\bul}(\de_1^*)$ to $U^*$ is an isomorphism of vector spaces
$$
U^*  ~\cong ~ H^{\bul}(\de_1^*) \big(U^*\big) \subset H^{n-1}(V_\circ(n)^*, \delta^*_0). 
$$

Therefore, by duality,  the composition of $H^{\bul}(\de_1)$ with the 
projection 
$$
H^{2-n}(V_{\bul}(n), \delta_0) ~\to~ U  
$$
gives us an isomorphism of vector spaces
$$
H^{1-n} (V_\circ(n), \delta_0) \big/ \ker(H^{\bul}(\de_1)) ~\cong~ U.
$$

Thus the desired isomorphism \eqref{desired} follows and the proof of 
Theorem \ref{thm:H-Br} is complete. 

%
%

\appendix

\section{Verification of the Gerstenhaber relations}
\label{app:Ger-relation}
As above, $T_{\{a_1, a_2 \}}$ and $T_{a_1 a_2}$ denote the following vectors in $\Br(2)$: 
$$
T_{\{a_1, a_2 \}} : = T_{1\mbox{-}2} +  T_{2\mbox{-}1}, 
\qquad  
T_{a_1 a_2} : = \frac{1}{2}(T_{\cup} + T^{\opp}_{\cup}),
$$
where $T_{1\mbox{-}2}$, $T_{2\mbox{-}1}$, $T_{\cup}$, and $T^{\opp}_{\cup}$ 
are the brace trees shown in figure \ref{fig:examples}. 

The goal of this appendix is to prove the following statement. 
%
%
\begin{claim}
\label{cl:Ger-relations}
The vector  $T_{\{a_1, a_2 \}}$ satisfies the Jacobi identity 
\begin{equation}
\label{Jac}
T_{\{a_1, a_2 \}} \circ_1 T_{\{a_1, a_2 \}} + (1,2,3) \big(  T_{\{a_1, a_2 \}} \circ_1 T_{\{a_1, a_2 \}} \big) + 
(1,3,2) \big( T_{\{a_1, a_2 \}} \circ_1 T_{\{a_1, a_2 \}} \big) = 0
\end{equation}
and the vector $T_{a_1 a_2}$ fulfills these properties: 
\begin{equation}
\label{assoc-homot}
T_{a_1 a_2} \circ_1 T_{a_1 a_2} - T_{a_1 a_2} \circ_2 T_{a_1 a_2} \in \textrm{Im}(\de)
\end{equation}
\begin{equation}
\label{Leibniz-homot}
T_{\{a_1, a_2\}} \circ_2 T_{a_1 a_2} - T_{a_1 a_2} \circ_1 T_{\{a_1, a_2\}} - 
(1,2) \big( T_{a_1 a_2} \circ_2 T_{\{a_1, a_2\}} \big)  \in \textrm{Im}(\de). 
\end{equation}
\end{claim}

\begin{proof}
The insertion $T_{\{a_1, a_2 \}} \circ_1 T_{\{a_1, a_2 \}}$ is computed explicitly 
in figure \ref{fig:Jac}. It is clear that the sum over the cyclic permutations
of the first term (resp. the third term) will cancel the sum over the cyclic permutations
of the sixth term (resp. the forth term). Similarly, the sum over the cyclic permutations 
of the second term (resp. the fifth term) cancels the sum over the cyclic permutations of 
the seventh term (resp. the eighth term). Thus identity \eqref{Jac} holds. 
%
%
\begin{figure}[htp] 
\centering
\begin{minipage}[t]{\linewidth}
\centering 
\begin{tikzpicture}
[lab/.style={draw, circle, minimum size=5, inner sep=1}, 
n/.style={draw, circle, fill, minimum size=5, inner sep=1}, scale=.7]
\draw  (-1,0) node { $\Big($};
\draw  (0,-0.7) node[root] (rr) {};
\draw (0,0) node[lab] (v1) {$\scriptstyle 1$};
\draw (0,0.7) node[lab] (v2) {$\scriptstyle 2$};
\draw (v1) edge (rr) edge (v2);
\draw  (1,0) node { $+$};
\draw  (2,-0.7) node[root] (rrr) {};
\draw (2,0) node[lab] (vv2) {$\scriptstyle 2$};
\draw (2,0.7) node[lab] (vv1) {$\scriptstyle 1$};
\draw (vv2) edge (rrr) edge (vv1);
\draw  (3,0) node { $\Big)$};
\draw  (3.7,0) node { $\circ_1$};
\begin{scope}[shift={(5.5,0)}]
\draw  (-1,0) node { $\Big($};
\draw  (0,-0.7) node[root] (rr) {};
\draw (0,0) node[lab] (v1) {$\scriptstyle 1$};
\draw (0,0.7) node[lab] (v2) {$\scriptstyle 2$};
\draw (v1) edge (rr) edge (v2);
\draw  (1,0) node { $+$};
\draw  (2,-0.7) node[root] (rrr) {};
\draw (2,0) node[lab] (vv2) {$\scriptstyle 2$};
\draw (2,0.7) node[lab] (vv1) {$\scriptstyle 1$};
\draw (vv2) edge (rrr) edge (vv1);
\draw  (3,0) node { $\Big)$};
\end{scope}
\draw  (9.5,0) node { $=$};
\end{tikzpicture}
\end{minipage}
\begin{minipage}[t]{\linewidth}
\vspace{0.5cm}
\end{minipage}
%
%
\begin{minipage}[t]{\linewidth}
\centering 
\begin{tikzpicture}
[lab/.style={draw, circle, minimum size=5, inner sep=1}, 
n/.style={draw, circle, fill, minimum size=5, inner sep=1}, scale=.7]
\draw  (0,-0.7) node[root] (rr) {};
\draw (0,0) node[lab] (v1) {$\scriptstyle 1$};
\draw (-0.5,0.7) node[lab] (v3) {$\scriptstyle 3$};
\draw (0.5,0.7) node[lab] (v2) {$\scriptstyle 2$};
\draw (v1) edge (rr) edge (v3) edge (v2);
\draw  (1.5,0) node { $-$};
\begin{scope}[shift={(2.8,0)}]
\draw  (0,-0.7) node[root] (rr) {};
\draw (0,0) node[lab] (v1) {$\scriptstyle 1$};
\draw (0,0.8) node[lab] (v2) {$\scriptstyle 2$};
\draw (0,1.6) node[lab] (v3) {$\scriptstyle 3$};
\draw (v1) edge (rr) edge (v2) (v2) edge (v3);
\end{scope}
\draw  (4,0) node { $-$};
\begin{scope}[shift={(5.5,0)}]
\draw  (0,-0.7) node[root] (rr) {};
\draw (0,0) node[lab] (v1) {$\scriptstyle 1$};
\draw (0.5,0.7) node[lab] (v3) {$\scriptstyle 3$};
\draw (-0.5,0.7) node[lab] (v2) {$\scriptstyle 2$};
\draw (v1) edge (rr) edge (v3) edge (v2);
\end{scope}
\draw  (7,0) node { $+$};
\begin{scope}[shift={(8.5,0)}]
\draw  (0,-0.7) node[root] (rr) {};
\draw (0,0) node[lab] (v2) {$\scriptstyle 2$};
\draw (-0.5,0.7) node[lab] (v3) {$\scriptstyle 3$};
\draw (0.5,0.7) node[lab] (v1) {$\scriptstyle 1$};
\draw (v2) edge (rr) edge (v3) edge (v1);
\end{scope}
\draw  (10,0) node { $-$};
\begin{scope}[shift={(11,0)}]
\draw  (0,-0.7) node[root] (rr) {};
\draw (0,0) node[lab] (v2) {$\scriptstyle 2$};
\draw (0,0.8) node[lab] (v1) {$\scriptstyle 1$};
\draw (0,1.6) node[lab] (v3) {$\scriptstyle 3$};
\draw (v2) edge (rr) edge (v1) (v1) edge (v3);
\end{scope}
\draw  (12,0) node { $-$};
\begin{scope}[shift={(13,0)}]
\draw  (0,-0.7) node[root] (rr) {};
\draw (0,0) node[lab] (v2) {$\scriptstyle 2$};
\draw (-0.5,0.7) node[lab] (v1) {$\scriptstyle 1$};
\draw (0.5,0.7) node[lab] (v3) {$\scriptstyle 3$};
\draw (v2) edge (rr) edge (v3) edge (v1);
\end{scope}
\draw  (14.3,0) node { $+$};
\begin{scope}[shift={(15.5,0)}]
\draw  (0,-0.7) node[root] (rr) {};
\draw (0,0) node[lab] (v3) {$\scriptstyle 3$};
\draw (0,0.8) node[lab] (v1) {$\scriptstyle 1$};
\draw (0,1.6) node[lab] (v2) {$\scriptstyle 2$};
\draw (v3) edge (rr) edge (v1) (v1) edge (v2);
\end{scope}
\draw  (16.5,0) node { $+$};
\begin{scope}[shift={(17.5,0)}]
\draw  (0,-0.7) node[root] (rr) {};
\draw (0,0) node[lab] (v3) {$\scriptstyle 3$};
\draw (0,0.8) node[lab] (v2) {$\scriptstyle 2$};
\draw (0,1.6) node[lab] (v1) {$\scriptstyle 1$};
\draw (v3) edge (rr) edge (v2) (v2) edge (v1);
\end{scope}
\end{tikzpicture}
\end{minipage}
\caption{Computation of the vector $T_{\{a_1, a_2 \}} \circ_1 T_{\{a_1, a_2 \}} \in \Br(3)$} \label{fig:Jac}
\end{figure}
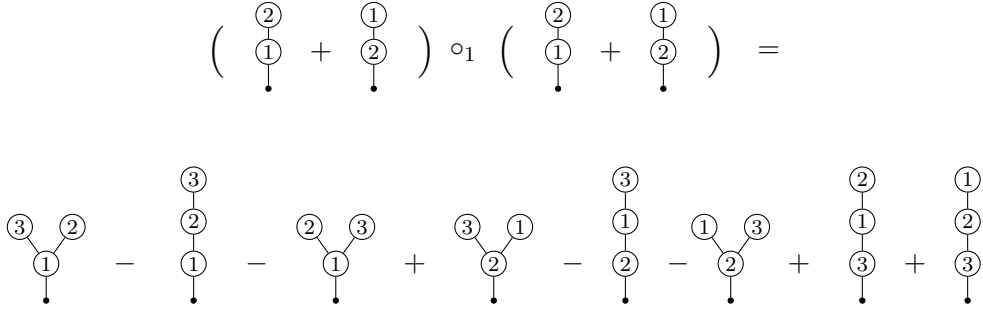

A simple computation shows that 
\begin{equation}
\label{de-T-1-2}
\de  (T_{1\mbox{-}2}) = T_{\cup} - T_{\cup}^{opp}. 
\end{equation}
Hence
\begin{equation}
\label{Ta1a2-cup}
T_{a_1 a_2} = T_{\cup} - \frac{1}{2} \de (T_{1\mbox{-}2}).
\end{equation}

On the other hand, 
$$
\begin{tikzpicture}
[lab/.style={draw, circle, minimum size=5, inner sep=1}, 
n/.style={draw, circle, fill, minimum size=5, inner sep=1}, scale=.7]
\draw  (-1.5,0) node { $\de$};
\draw  (0,-0.7) node[root] (rr) {};
\draw (0,0) node[n] (n) {};
\draw (-0.7,0.7) node[lab] (v1) {$\scriptstyle 1$};
\draw (0,0.7) node[lab] (v2) {$\scriptstyle 2$};
\draw (0.7,0.7) node[lab] (v3) {$\scriptstyle 3$};
\draw (n) edge (rr) edge (v1) edge (v2) edge (v3);
\draw  (4.5,0) node { $= \quad T_{\cup} \circ_1 T_{\cup} ~ - ~ T_{\cup} \circ_2 T_{\cup} $};
\end{tikzpicture}
$$
Therefore, the vector 
$$
T_{a_1 a_2} \circ_1 T_{a_1 a_2} - T_{a_1 a_2} \circ_2 T_{a_1 a_2} 
$$
indeed belongs to $\textrm{Im}(\de)$, i.e. \eqref{assoc-homot} holds. 

To prove \eqref{Leibniz-homot}, we denote by $T_{1\mbox{-}(2,3)}$ the 
following brace tree:
$$
\begin{tikzpicture}
[lab/.style={draw, circle, minimum size=5, inner sep=1}, 
n/.style={draw, circle, fill, minimum size=5, inner sep=1}, scale=.7]
\draw  (-3,0) node { $T_{1\mbox{-}(2,3)} ~~:\,=$};
\draw  (0,-0.7) node[root] (rr) {};
\draw (0,0) node[lab] (v1) {$\scriptstyle 1$};
\draw (-0.5,0.7) node[lab] (v2) {$\scriptstyle 2$};
\draw (0.5,0.7) node[lab] (v3) {$\scriptstyle 3$};
\draw (v1) edge (rr) edge (v2) edge (v3);
\end{tikzpicture}
$$

We compute the differential $\de (T_{1\mbox{-}(2,3)})$ in figure \ref{fig:diff-T1-23}
%
%
\begin{figure}[htp] 
\centering
\begin{tikzpicture}
[lab/.style={draw, circle, minimum size=5, inner sep=1}, 
n/.style={draw, circle, fill, minimum size=5, inner sep=1}, scale=.7]
\draw  (-3,0) node { $\de (T_{1\mbox{-}(2,3)}) ~ =$};
\draw  (0,-0.7) node[root] (rr) {};
\draw (0,0) node[lab] (v1) {$\scriptstyle 1$};
\draw (0,0.8) node[n] (n) {};
\draw (-0.5,1.5) node[lab] (v2) {$\scriptstyle 2$};
\draw (0.5,1.5) node[lab] (v3) {$\scriptstyle 3$};
\draw (v1) edge (rr) edge (n) (n) edge (v2)  edge (v3);
\draw  (1.5,0) node { $+$};
\begin{scope}[shift={(3,0)}]
\draw  (0,-0.7) node[root] (rr) {};
\draw (0,0) node[n] (n) {};
\draw (-0.7,0.7) node[lab] (v1) {$\scriptstyle 1$};
\draw (0,0.7) node[lab] (v2) {$\scriptstyle 2$};
\draw (0.7,0.7) node[lab] (v3) {$\scriptstyle 3$};
\draw (n) edge (rr) edge (v1) edge (v2)  edge (v3);
\end{scope}
\draw  (4.5,0) node { $+$};
\begin{scope}[shift={(6,0)}]
\draw  (0,-0.7) node[root] (rr) {};
\draw (0,0) node[n] (n) {};
\draw (-0.5,0.7) node[lab] (v1) {$\scriptstyle 1$};
\draw (0.5,0.7) node[lab] (v3) {$\scriptstyle 3$};
\draw (-0.5,1.5) node[lab] (v2) {$\scriptstyle 2$};
\draw (n) edge (rr) edge (v1) edge (v3) (v1) edge (v2);
\end{scope}
\draw  (7.5,0) node { $-$};
\begin{scope}[shift={(9,0)}]
\draw  (0,-0.7) node[root] (rr) {};
\draw (0,0) node[n] (n) {};
\draw (-0.7,0.7) node[lab] (v2) {$\scriptstyle 2$};
\draw (0,0.7) node[lab] (v1) {$\scriptstyle 1$};
\draw (0.7,0.7) node[lab] (v3) {$\scriptstyle 3$};
\draw (n) edge (rr) edge (v1) edge (v2)  edge (v3);
\end{scope}
\draw  (10.5,0) node { $-$};
\begin{scope}[shift={(12,0)}]
\draw  (0,-0.7) node[root] (rr) {};
\draw (0,0) node[n] (n) {};
\draw (-0.5,0.7) node[lab] (v2) {$\scriptstyle 2$};
\draw (0.5,0.7) node[lab] (v1) {$\scriptstyle 1$};
\draw (0.5,1.5) node[lab] (v3) {$\scriptstyle 3$};
\draw (n) edge (rr) edge (v1) edge (v2) (v1) edge (v3);
\end{scope}
\draw  (13.5,0) node { $+$};
\begin{scope}[shift={(15,0)}]
\draw  (0,-0.7) node[root] (rr) {};
\draw (0,0) node[n] (n) {};
\draw (-0.7,0.7) node[lab] (v2) {$\scriptstyle 2$};
\draw (0,0.7) node[lab] (v3) {$\scriptstyle 3$};
\draw (0.7,0.7) node[lab] (v1) {$\scriptstyle 1$};
\draw (n) edge (rr) edge (v1) edge (v2)  edge (v3);
\end{scope}
\end{tikzpicture}
\caption{Computation of the differential $\de(T_{1\mbox{-}(2,3)})$} \label{fig:diff-T1-23}
\end{figure}

The insertions $T_{\{a_1, a_2\}} \circ_2 T_{\cup}$ and $T_{\cup} \circ_1 T_{\{a_1, a_2\}}$ 
are computed in figures \ref{fig:brack-circ2-cup} and \ref{fig:cup-circ1-brack}, respectively, 
and the vector $(1,2) \big( T_{\cup} \circ_2  T_{\{a_1, a_2\}} \big) $ is shown in figure
\ref{fig:cup-circ2-brack}. 

%
%
\begin{figure}[htp] 
\centering
\begin{tikzpicture}
[lab/.style={draw, circle, minimum size=5, inner sep=1}, 
n/.style={draw, circle, fill, minimum size=5, inner sep=1}, scale=.7]
\draw  (-3,0) node { $T_{\{a_1, a_2\}} \circ_2 T_{\cup} ~~ =$};
\draw  (0,-0.7) node[root] (rr) {};
\draw (0,0) node[lab] (v1) {$\scriptstyle 1$};
\draw (0,0.8) node[n] (n) {};
\draw (-0.5,1.5) node[lab] (v2) {$\scriptstyle 2$};
\draw (0.5,1.5) node[lab] (v3) {$\scriptstyle 3$};
\draw (v1) edge (rr) edge (n) (n) edge (v2)  edge (v3);
\draw  (1.5,0) node { $+$};
\begin{scope}[shift={(3,0)}]
\draw  (0,-0.7) node[root] (rr) {};
\draw (0,0) node[n] (n) {};
\draw (-0.7,0.7) node[lab] (v1) {$\scriptstyle 1$};
\draw (0,0.7) node[lab] (v2) {$\scriptstyle 2$};
\draw (0.7,0.7) node[lab] (v3) {$\scriptstyle 3$};
\draw (n) edge (rr) edge (v1) edge (v2)  edge (v3);
\end{scope}
\draw  (4.5,0) node { $-$};
\begin{scope}[shift={(6,0)}]
\draw  (0,-0.7) node[root] (rr) {};
\draw (0,0) node[n] (n) {};
\draw (-0.5,0.7) node[lab] (v2) {$\scriptstyle 2$};
\draw (0.5,0.7) node[lab] (v3) {$\scriptstyle 3$};
\draw (-0.5,1.5) node[lab] (v1) {$\scriptstyle 1$};
\draw (n) edge (rr) edge (v2)  edge (v3) (v2) edge (v1);
\end{scope}
\draw  (7.5,0) node { $-$};
\begin{scope}[shift={(9,0)}]
\draw  (0,-0.7) node[root] (rr) {};
\draw (0,0) node[n] (n) {};
\draw (-0.7,0.7) node[lab] (v2) {$\scriptstyle 2$};
\draw (0,0.7) node[lab] (v1) {$\scriptstyle 1$};
\draw (0.7,0.7) node[lab] (v3) {$\scriptstyle 3$};
\draw (n) edge (rr) edge (v1) edge (v2)  edge (v3);
\end{scope}
\draw  (10.5,0) node { $+$};
\begin{scope}[shift={(12,0)}]
\draw  (0,-0.7) node[root] (rr) {};
\draw (0,0) node[n] (n) {};
\draw (-0.5,0.7) node[lab] (v2) {$\scriptstyle 2$};
\draw (0.5,0.7) node[lab] (v3) {$\scriptstyle 3$};
\draw (0.5,1.5) node[lab] (v1) {$\scriptstyle 1$};
\draw (n) edge (rr) edge (v2)  edge (v3) (v3) edge (v1);
\end{scope}
\draw  (13.5,0) node { $+$};
\begin{scope}[shift={(15,0)}]
\draw  (0,-0.7) node[root] (rr) {};
\draw (0,0) node[n] (n) {};
\draw (-0.7,0.7) node[lab] (v2) {$\scriptstyle 2$};
\draw (0,0.7) node[lab] (v3) {$\scriptstyle 3$};
\draw (0.7,0.7) node[lab] (v1) {$\scriptstyle 1$};
\draw (n) edge (rr) edge (v1) edge (v2)  edge (v3);
\end{scope}
\end{tikzpicture}
\caption{Computation of the insertion $T_{\{a_1, a_2\}} \circ_2 T_{\cup}$ } \label{fig:brack-circ2-cup}
\end{figure}
%
%
\begin{figure}[htp] 
\centering
\begin{tikzpicture}
[lab/.style={draw, circle, minimum size=5, inner sep=1}, 
n/.style={draw, circle, fill, minimum size=5, inner sep=1}, scale=.7]
\draw  (-3.5,0) node { $T_{\cup}  \circ_1 T_{\{a_1, a_2\}} ~~~ = ~~~ -$};
\draw  (0,-0.7) node[root] (rr) {};
\draw (0,0) node[n] (n) {};
\draw (-0.5,0.7) node[lab] (v1) {$\scriptstyle 1$};
\draw (0.5,0.7) node[lab] (v3) {$\scriptstyle 3$};
\draw (-0.5,1.5) node[lab] (v2) {$\scriptstyle 2$};
\draw (n) edge (rr) edge (v1) edge (v3) (v1)  edge (v2);
\draw  (1.5,0) node { $-$};
\begin{scope}[shift={(3,0)}]
\draw  (0,-0.7) node[root] (rr) {};
\draw (0,0) node[n] (n) {};
\draw (-0.5,0.7) node[lab] (v2) {$\scriptstyle 2$};
\draw (0.5,0.7) node[lab] (v3) {$\scriptstyle 3$};
\draw (-0.5,1.5) node[lab] (v1) {$\scriptstyle 1$};
\draw (n) edge (rr) edge (v2) edge (v3) (v2)  edge (v1);
\end{scope}
\end{tikzpicture}
\caption{Computation of the insertion $T_{\cup}  \circ_1 T_{\{a_1, a_2\}}$ } \label{fig:cup-circ1-brack}
\end{figure}
%
%
\begin{figure}[htp] 
\centering
\begin{tikzpicture}
[lab/.style={draw, circle, minimum size=5, inner sep=1}, 
n/.style={draw, circle, fill, minimum size=5, inner sep=1}, scale=.7]
\draw  (-4,0) node { $(1,2) \big( T_{\cup}  \circ_2 T_{\{a_1, a_2\} } \big) ~~ =$};
\draw  (0,-0.7) node[root] (rr) {};
\draw (0,0) node[n] (n) {};
\draw (-0.5,0.7) node[lab] (v2) {$\scriptstyle 2$};
\draw (0.5,0.7) node[lab] (v1) {$\scriptstyle 1$};
\draw (0.5,1.5) node[lab] (v3) {$\scriptstyle 3$};
\draw (n) edge (rr) edge (v1) edge (v2) (v1)  edge (v3);
\draw  (1.5,0) node { $+$};
\begin{scope}[shift={(3,0)}]
\draw  (0,-0.7) node[root] (rr) {};
\draw (0,0) node[n] (n) {};
\draw (-0.5,0.7) node[lab] (v2) {$\scriptstyle 2$};
\draw (0.5,0.7) node[lab] (v3) {$\scriptstyle 3$};
\draw (0.5,1.5) node[lab] (v1) {$\scriptstyle 1$};
\draw (n) edge (rr) edge (v2) edge (v3) (v3)  edge (v1);
\end{scope}
\end{tikzpicture}
\caption{The vector $(1,2) \big( T_{\cup}  \circ_2 T_{\{a_1, a_2\} } \big) $ } \label{fig:cup-circ2-brack}
\end{figure}

Adding all these expressions and performing obvious cancelations, we conclude that 
\begin{equation}
\label{Leibniz-cup}
T_{\{a_1, a_2\}} \circ_2 T_{\cup} - T_{\cup}  \circ_1 T_{\{a_1, a_2\}} - (1,2) \big( T_{\cup}  \circ_2 T_{\{a_1, a_2\} } \big)  = \de(T_{1\mbox{-}(2,3)}). 
\end{equation}
Finally, combining \eqref{Ta1a2-cup} with \eqref{Leibniz-cup}, we deduce \eqref{Leibniz-homot}. 

Claim \ref{cl:Ger-relations} is proved. 
\end{proof}

\section{The spectral sequence for $(V_\bullet(n), \de_0)$ degenerates at the second page}
\label{app:E2-Einfty}

Let us study in a bit more detail the dual of the map $\mj : \La\Lie\to\Br$. 
In arity $n$ the dual map can be realized as a composition 
\[
 \Br^*(n) \to \cT^*(n) \to \La^{-1} \coAs(n)\to \La^{-1}\coLie(n)
\]
where we use the following objects and morphisms:
\begin{itemize}
\item $\cT(n) \subset \Br(n)$ is the graded subspace of trees without neutral vertices. 
 (In fact, the $\cT(n)$ assemble to form an operad whose twist is essentially $\Br$, cf. \cite{DW}.)
 \item The map $\Br^*(n) \to \cT^*(n)$ is the natural projection. (Concretely, it sends graphs with neutral vertices to zero.)
 \item The map $\La^{-1} \coAs(n)\to \La^{-1}\coLie(n)$ is the natural projection arising from the inclusion $\Lie\to\As$.
 Note that we may identify $\La^{-1} \coAs(n)$ (up to a degree shift) with the subspace of the space of words
 \[
  \bbK\langle X_1,\dots,X_n\rangle
 \]
in formal odd variables, each appearing exactly once. 
The space $\bbK\langle X_1,\dots,X_n\rangle$ is a $\mathbb{Z}^n$ graded augmented commutative algebra with the shuffle product $\bullet_{sh}$ and unit the empty word.
We denote by $A_n\subset \bbK\langle X_1,\dots,X_n\rangle$ the augmentation ideal.
The space $\La^{-1}\coLie(n)$ may then be identified with the degree $(1,\dots,1)$-subspace of the quotient
\[
 A_n / (A_n \bullet_{sh} A_n).
\]
In this language, $\La^{-1} \coAs(n)\to \La^{-1}\coLie(n)$ is just the map induced on the degree $(1,\dots,1)$-subspaces of the obvious projection 
\[
 A_n \to A_n / (A_n \bullet_{sh} A_n).
\]
\item The map $f: \cT^*(n) \to \La^{-1} \coAs(n)\cong A_n^{(1,\dots,1)}$ can be defined recursively as follows.
If $n=1$ and $T\in \cT^*(1)$ is the unique tree with one vertex labelled 1, we set 
\[
 f(T) = X_1.
\]
If $n>1$ and $T\in \cT^*(n)$ is the tree with lowest vertex $j$, having children (in this order) $T_1,\dots, T_k$, we set recursively
\[
 f(T) = X_j(f(T_1)\bullet_{sh}\cdots \bullet_{sh} f(T_k)).
\]
\end{itemize}
For example, if $\la \in S_n$ and $T^n_{\la}$ is the brace tree shown in figure \ref{fig:T-la}, then 
$$
f(T^n_{\la}) = X_{\la(1)}  X_{\la(2)} \dots  X_{\la(n)}\,.
$$
Furthermore, if 
$$
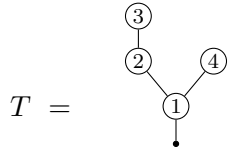
\begin{tikzpicture}
\draw  (-1.8,0) node {$T~=$};
\node[root] (r) at (0, -0.5) {};
\draw (0,0) node[lab] (v1) { $\scriptstyle 1$};
\draw (-0.5,0.6) node[lab] (v2) { $\scriptstyle 2$};
\draw (-0.5,1.2) node[lab] (v3) { $\scriptstyle 3$};
\draw (0.5,0.6) node[lab] (v4) { $\scriptstyle 4$};
\draw (v1) edge (v2) edge (v4) edge (r) (v2) edge (v3);
\end{tikzpicture} 
$$
then 
$$
f(T) = X_1 \big( (X_2 X_3) \bullet_{sh} X_4\big) = X_1 (X_2 X_3 X_4  - X_2 X_4 X_3 + X_4 X_2 X_3). 
$$

The composition $g:\Br^*(n) \to \cT^*(n) \stackrel{f}{\to} \La^{-1} \coAs(n)$ appearing above is of interest in its own right.
It is does not commute with the differential, i.e., $g\circ \delta^*\neq 0$. However, we claim that 
\begin{lemma}
\label{lem:gcircdelta}
For every brace tree $T$
$$
g\circ \delta_0^* (T) = 0.
$$
\end{lemma}
\begin{proof}
It is clear that we should only consider $g\circ \delta_0^*(T)$ for a brace tree $T$ with 
exactly one neutral vertex which is not in the lowest possible position. 

Up to an overall sign factor, the differential $\de_0^*$ turns the branch 
$$
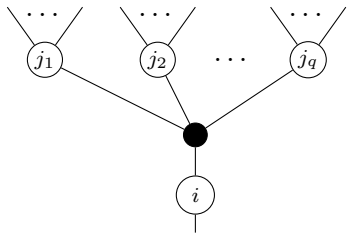
\begin{tikzpicture}
\draw (0,0) node[lab] (vi) { $\scriptstyle ~i~$};
\draw (0,0.8) node[n] (n) {$\scriptstyle n$};
\draw (-2,1.8) node[lab] (vj1) { $\scriptstyle j_{1}$};
\draw (-0.5,1.8) node[lab] (vj2) { $\scriptstyle j_{2}$};
\draw  (0.5,1.8) node {$\dots$};
\draw (1.5,1.8) node[lab] (vjq) { $\scriptstyle j_{q}$};
\draw  (vi) edge (0,-0.5) (n) edge (vi) edge (vj1) edge (vj2) edge (vjq);
\draw  (vj1) edge (-2.5,2.5) edge (-1.5, 2.5); \draw  (-2, 2.4) node {$\dots$};
\draw  (vj2) edge (-1,2.5) edge (0, 2.5); \draw  (-0.5, 2.4) node {$\dots$};
\draw  (vjq) edge (1,2.5) edge (2, 2.5); \draw  (1.5, 2.4) node {$\dots$};
\end{tikzpicture} 
$$
into the linear combination
$$
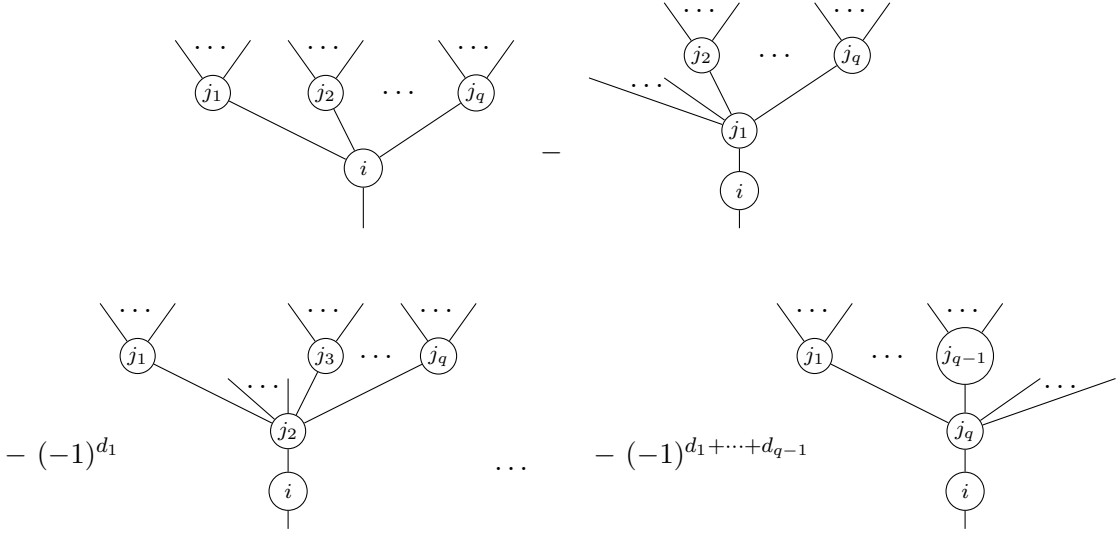
\begin{tikzpicture}
\begin{scope}[shift={(1,0)}]
\draw (0,0.3) node[lab] (vi) { $\scriptstyle ~i~$};
\draw (-2,1.3) node[lab] (vj1) { $\scriptstyle j_{1}$};
\draw (-0.5,1.3) node[lab] (vj2) { $\scriptstyle j_{2}$};
\draw  (0.5,1.3) node {$\dots$};
\draw (1.5,1.3) node[lab] (vjq) { $\scriptstyle j_{q}$};
\draw  (vi) edge (0,-0.5) edge (vj1) edge (vj2) edge (vjq);
\draw  (vj1) edge (-2.5,2) edge (-1.5, 2); \draw  (-2, 1.9) node {$\dots$};
\draw  (vj2) edge (-1,2) edge (0, 2); \draw  (-0.5, 1.9) node {$\dots$};
\draw  (vjq) edge (1,2) edge (2, 2); \draw  (1.5, 1.9) node {$\dots$};
\draw  (2.5,0.5) node {$ - $};
\begin{scope}[shift={(5,0)}]
\draw (0,0) node[lab] (vi) { $\scriptstyle ~i~$};
\draw (0,0.8) node[lab] (vj1) {$\scriptstyle j_1$};
\draw (-0.5,1.8) node[lab] (vj2) { $\scriptstyle j_{2}$};
\draw  (0.5,1.8) node {$\dots$};
\draw (1.5,1.8) node[lab] (vjq) { $\scriptstyle j_{q}$};
\draw  (vi) edge (0,-0.5) (vj1) edge (vi) edge (vj2) edge (vjq);
\draw  (vj1) edge (-2, 1.5) edge (-1, 1.5); \draw  (-1.2, 1.4) node {$\dots$};
\draw  (vj2) edge (-1,2.5) edge (0, 2.5); \draw  (-0.5, 2.4) node {$\dots$};
\draw  (vjq) edge (1,2.5) edge (2, 2.5); \draw  (1.5, 2.4) node {$\dots$};
\end{scope}
\end{scope}
\begin{scope}[shift={(0,-4)}]
\draw  (-3,0.5) node {$ -~(-1)^{d_1}$};
\draw (0,0) node[lab] (vi) { $\scriptstyle ~i~$};
\draw (0,0.8) node[lab] (vj2) {$\scriptstyle j_2$};

\draw (-2,1.8) node[lab] (vj1) { $\scriptstyle j_{1}$};

\draw  (1.2,1.8) node {$\dots$};
\draw (2,1.8) node[lab] (vjq) { $\scriptstyle j_{q}$};
\draw (0.5,1.8) node[lab] (vj3) { $\scriptstyle j_{3}$};
\draw  (vi) edge (0,-0.5) (vj2) edge (vi) edge (vj1) edge (vj3) edge (vjq);
\draw  (vj1) edge (-2.5,2.5) edge (-1.5, 2.5); \draw  (-2, 2.4) node {$\dots$};
\draw  (vj2) edge (-0.8,1.5) edge (0, 1.5); \draw  (-0.3, 1.4) node {$\dots$};
\draw  (vj3) edge (0,2.5) edge (1, 2.5); \draw  (0.5, 2.4) node {$\dots$};
\draw  (vjq) edge (1.5,2.5) edge (2.5, 2.5); \draw  (2, 2.4) node {$\dots$};
\end{scope}
\begin{scope}[shift={(9,-4)}]
\draw  (-6,0.3) node {$\dots$};
\draw  (-3.5,0.5) node {$ -~(-1)^{d_1 + \dots + d_{q-1}}$};
\draw (0,0) node[lab] (vi) { $\scriptstyle ~i~$};
\draw (0,0.8) node[lab] (vjq) {$\scriptstyle j_q$};
\draw (-2,1.8) node[lab] (vj1) { $\scriptstyle j_{1}$};
\draw  (-1,1.8) node {$\dots$};
\draw (0,1.8) node[lab] (vj1q) { $\scriptstyle j_{q-1}$};
\draw  (vi) edge (0,-0.5) (vjq) edge (vi) edge (vj1) edge (vj1q);
\draw  (vj1) edge (-2.5,2.5) edge (-1.5, 2.5); \draw  (-2, 2.4) node {$\dots$};
\draw  (vj1q) edge (-0.5,2.5) edge (0.5, 2.5); \draw  (0, 2.4) node {$\dots$};
\draw  (vjq) edge (1,1.5) edge (2, 1.5); \draw  (1.3, 1.4) node {$\dots$};
\end{scope}
\end{tikzpicture} 
$$
where $d_k$ is the degree of the brach which originates from the neutral vertex 
and contains vertex $j_k$.  

Therefore $g\circ \delta_0^*(T)$ contains this expression 
\begin{equation}
\label{long-sum-shuff}
X_i \big( f_{j_1} \bullet_{sh} f_{j_2} \bullet_{sh} \dots  \bullet_{sh} f_{j_q} \big) -
X_i X_{j_1} \big( h_{j_1} \bullet_{sh}  f_{j_2} \bullet_{sh} \dots  \bullet_{sh} f_{j_q}  \big)
\end{equation}
$$
-(-1)^{d_1} X_i X_{j_2} \big( f_{j_1} \bullet_{sh} h_{j_2} \bullet_{sh} 
f_{j_3} \bullet_{sh} \dots \bullet_{sh} f_{j_q} \big) - \dots
$$
$$
-(-1)^{d_1+ d_2 + \dots + d_{q-1}} X_i X_{j_q} \big( f_{j_1} \bullet_{sh} \dots \bullet_{sh} f_{j_{q-1}} \bullet_{sh} h_{j_q}  \big)
$$
as a factor. Here $f_{j_k}$ is the value of $f$ on the brach which originates at the neutral vertex 
and contains vertex $j_k$, while
$$
h_{j_k} = f(b_{j_k 1}) \bullet_{sh} f (b_{j_k 2}) \bullet_{sh} \dots \bullet_{sh} f(b_{j_k r_k}), 
$$
where $b_{j_k t}$ is the $t$-th brach which originates from vertex $j_k$. 

Using the definition of the shuffle product, it is easy to see that 
the expression \eqref{long-sum-shuff} is zero. 

Thus the lemma follows. 
\end{proof}

\begin{remark}
\label{rem:gshuffle}
Let us observe that the map $g\circ \delta_1^*$ has the following nice combinatorial description:
If $T\in \Br^*(n)$ is a brace tree, then $g\circ \delta_1^*(T)=0$ unless $T$ has exactly 
one neutral vertex, which is the lowest vertex.
In this case 
$$
g \circ \delta_1^*(T) = f(T_1)\bullet_{sh}\cdots \bullet_{sh}f(T_k),
$$
where $T_1, \dots,T_k$ are the branches which originate at the neutral vertex. 

On the other hand, Lemma \ref{lem:gcircdelta} implies that 
$g\circ \delta^*= g\circ \delta_1^*$. Thus  $g\circ \delta^*(T)=0$ unless $T$ has exactly 
one neutral vertex, which is the lowest vertex and, in this case, 
\begin{equation}
\label{g-delta-star}
g \circ \delta^*(T) = f(T_1)\bullet_{sh}\cdots \bullet_{sh}f(T_k).
\end{equation}
\end{remark}

Let us now consider the dual cochain complex
$$
(V_\bullet(n)^* , \de^*_0)
$$ 
and construct a set of vectors in the top degree $n-2$ which 
will play an important role. 

Let $k$ be an integer $\ge 2$ and $(r_1, r_2, \dots, r_k)$ be a tuple 
of positive integers such that $r_1 + r_2 + \dots + r_k  = n$. 
For every such tuple, we consider a brace trees 
$T^{\si}_{r_1, \dots, r_k}$ shown in figure \ref{fig:fork}, 
where $\si$ is a permutation in $S_n$
\begin{equation}
\label{si-nice-order}
\si =
\left(
\begin{array}{cccccccccccc}
1 & 2  & \dots & r_1 & r_1 +1 & \dots & r_1+r_2 & \dots & \dots &n -r_k+1 & \dots & n   \\
i_1^1 & i^1_2  & \dots & i^1_{r_1} & i^2_1 & \dots & i^2_{r_2} & \dots & \dots &  i^k_1  & \dots & i^k_{r_k}     
\end{array}
\right)
\end{equation}
which satisfies these properties\footnote{In particular, $i^1_1$ is necessarily $1$.}
\begin{equation}
\label{sigma-pty}
i^m_1 = \min \{ i^m_1, i^m_2, \dots, i^m_{r_m} \} \quad \forall~~ m, \qquad 
\textrm{and} \qquad
i^1_1 < i^2_1 < \dots < i^k_1.  
\end{equation}

 \begin{figure}[htp] 
\centering 
\begin{tikzpicture}
\node[root] (r) at (0,0) {};
\node[n] (rr) at (0,.5) {};
\node[ext] (t1) at (-1.5,1) {$i_1^1$};
\node[] (t12) at (-1.5,2) {$\vdots$};
\node[ext] (t13) at (-1.5,3) {$i_{r_1}^1$};
\node[ext] (t2) at (1.5,1) {$i_1^k$};
\node[] (t22) at (1.5,2) {$\vdots$};
\node[ext] (t23) at (1.5,3) {$i_{r_k}^k$};
\node[] (dd) at (0,1) {$\dots$};
\draw (rr) edge (r) edge (t1) edge (t2) edge (dd)
      (t12) edge (t1) edge (t13)
      (t22) edge (t2) edge (t23); 
\end{tikzpicture}
\caption{The brace tree $T^{\si}_{r_1, \dots, r_k}$} \label{fig:fork}
\end{figure}

Moreover, we set 
\begin{equation}
\label{Y-si}
Y^{\si}_{r_1, \dots, r_k} = \frac{1}{k!} \sum_{\tau \in S_k}  \tau_* (T^{\si}_{r_1, \dots, r_k}), 
\end{equation}
where $\tau_*$ rearranges the $k$ branches of $T^{\si}_{r_1, \dots, r_k}$ originating from 
the neutral vertex with the appropriate sign factor. For example, 
$$
\begin{tikzpicture}
\draw  (-2.5,0.5) node {$Y^{\si}_{r_1, r_2} ~ = ~ \dis \frac{1}{2}$};
\node[root] (r) at (0,-0.5) {};
\draw (0,0) node[n] (n) {$~$};
\draw (-0.5,0.7) node[lab] (vi11) { $\scriptstyle i^1_1$};
\draw  (-0.5,1.7) node {$\vdots$};
\draw (-0.5,2.5) node[lab] (vi1r1) { $\scriptstyle i^1_{r_1}$};
\draw (0.5,0.7) node[lab] (vi21) { $\scriptstyle i^2_1$};
\draw  (0.5,1.7) node {$\vdots$};
\draw (0.5,2.5) node[lab] (vi2r2) { $\scriptstyle i^2_{r_2}$};
\draw (n) edge (r)  edge (vi11) edge (vi21);
\draw (vi11) edge (-0.5, 1.2)  (vi21) edge (0.5, 1.2) (vi1r1) edge (-0.5, 1.9) (vi2r2) edge (0.5, 1.9);
\begin{scope}[shift={(4.5,0)}]
\draw  (-2,0.5) node {$+ ~ \dis \frac{(-1)^{r_1 r_2}}{2}$};
\node[root] (r) at (0,-0.5) {};
\draw (0,0) node[n] (n) {$~$};
\draw (0.5,0.7) node[lab] (vi11) { $\scriptstyle i^1_1$};
\draw  (0.5,1.7) node {$\vdots$};
\draw (0.5,2.5) node[lab] (vi1r1) { $\scriptstyle i^1_{r_1}$};
\draw (-0.5,0.7) node[lab] (vi21) { $\scriptstyle i^2_1$};
\draw  (-0.5,1.7) node {$\vdots$};
\draw (-0.5,2.5) node[lab] (vi2r2) { $\scriptstyle i^2_{r_2}$};
\draw (n) edge (r)  edge (vi11) edge (vi21);
\draw (vi11) edge (0.5, 1.2)  (vi21) edge (-0.5, 1.2) (vi1r1) edge (0.5, 1.9) (vi2r2) edge (-0.5, 1.9);
\end{scope}
\end{tikzpicture}
$$

We denote by $\Xi$ the set of all such vectors $Y^{\si}_{r_1, \dots, r_k}$ for all $k \ge 2$,
all tuples $(r_1, r_2, \dots, r_k)$, $r_1 + \dots + r_k =n$, and all permutations 
$\si$ satisfying \eqref{sigma-pty}. 
Due to the theorem about the cyclic decomposition of a permutation, it is clear 
that $\Xi$ has 
$$
n! - (n-1)!
$$ 
elements. Moreover, the subset $\Xi \subset V_\bullet(n)^*$ is linearly independent.  

Since every vector $Y^{\si}_{r_1, \dots, r_k}$ is in the top degree of 
$(V_\bullet(n)^* , \de^*_0)$, it is automatically a cocycle in this complex.  

Let us prove that 
\begin{claim}
\label{cl:forks-are-good}
Every non-trivial linear combination of vectors in $\Xi$ is a non-trivial 
cocycle in $(V_\bullet(n)^* , \de^*_0)$. 
\end{claim}
\begin{proof}
To prove this claim, we need Lemma \ref{lem:gcircdelta} and Remark \ref{rem:gshuffle}. 

Let us, first, prove that the map 
\begin{equation}
\label{g-delta-star-Forks}
(g\circ \delta^*) \Big|_{\vsspan_{\bbK}(\Xi)} :  \vsspan_{\bbK}(\Xi) \to \La^{-1}\coAs(n) \cong A_n^{(1,\dots,1)}
\end{equation}
is injective.  

Indeed, by Remark \ref{rem:gshuffle} and the symmetry of the shuffle product,
we have 
$$
g\circ \delta^* \big( Y^{\si}_{r_1, \dots, r_k} \big) = g\circ \delta^* \big( T^{\si}_{r_1, \dots, r_k} \big) = 
(X_{i^1_1} \dots X_{i^1_{r_1}}) \bullet_{sh}  \dots \bullet_{sh}  (X_{i^{k}_1} \dots X_{i^k_{r_k}}).
$$

Using the identification between $\La^{-1}\coLie(n)$ and the degree $(1,\dots,1)$-subspace 
of the quotient $A_n / (A_n \bullet_{sh} A_n)$, it is easy to see that $g\circ \delta^*$ gives us 
a surjective map from $ \vsspan_{\bbK}(\Xi)$ to the degree $(1,\dots,1)$-subspace of $A_n \bullet_{sh} A_n$. 
Since both  $ \vsspan_{\bbK}(\Xi)$
and the degree $(1,\dots,1)$-subspace of $A_n \bullet_{sh} A_n$ have the same dimension 
$$
n! - (n-1)!, 
$$
we conclude that \eqref{g-delta-star-Forks} is indeed injective. 

Let us consider a vector $v \in  \vsspan_{\bbK}(\Xi)$ and assume that 
$$
v = \de_0^* (w)
$$
for some $w \in V_\bullet(n)^*$. 

Using  Lemma \ref{lem:gcircdelta}, we conclude that
\[
 g( \delta^* v) = g( \delta^* \delta_0^*  w)= -g( \delta_0^* \delta^*  w) = 0.
\]

Thus, since \eqref{g-delta-star-Forks} is injective, we conclude that $v = 0$ 
and the desired claim follows. 
\end{proof}

With these preparations we are now ready to prove the following statement left open above.
\begin{lemma}
\label{lem:specseqabutment}
The spectral sequence arising in Section \ref{sec:H-V-bul} degenerates at the second page.
\end{lemma}
\begin{proof}
According to Claim \ref{cl:E2-V-bullet}, $E_2 V_\bullet(n)$ splits (as the graded 
vector space) into the direct sum 
$$
\Com \odot \La \Lie (n) \big/ \La \Lie (n)  ~\oplus ~ \bs \big( \La\Com \odot \La \Lie (n) \big/ \La \Lie (n) \big)\,.
$$
It is easy to see that every vector in the summand
\begin{equation}
\label{summand-low}
U : = \bs \big( \La\Com \odot \La \Lie (n) \big/ \La \Lie (n) \big)
\end{equation}
has degree $2-n$, while the summand  
\begin{equation}
\label{summand-more}
X := \Com \odot \La \Lie (n) \big/ \La \Lie (n)
\end{equation}
lives in degrees 
$$
2 - n \le {\scriptstyle \bullet} \le 0. 
$$

We also know that every vector in \eqref{summand-more} can be represented by 
a genuine cocycle in $\Br(n)$. Thus the restriction of all higher differentials $d_r$, $(r \ge 2)$
to the subspace \eqref{summand-more} is zero and it remains to show 
that the restriction of  $d_r$ for $r \ge 2$ to \eqref{summand-low} is also zero. 

To prove this statement, we pass to the obvious dual version of Claim \ref{cl:E2-V-bullet}, 
which says that
$$
E_2 V_\bullet(n)^* \cong X^* \oplus U^*,
$$
where $X^*$ is the kernel of the map $\Ger(n)^* \to \La \Lie(n)^*$ and $U^*$ is the linear 
dual of \eqref{summand-low}. 

The advantage of passing to the dual complex is that $U^*$ lives in the top degree $n-2$ 
of the cochain complex $(V_\bullet(n)^* , \de^*_0)$.
So all vectors in $U^*$ can be represented by genuine cocycles in $(V_\bullet(n)^* , \de^*_0)$.
Moreover, the first (potentially) non-zero differential $d_r^*,~ r \ge 2$ may only send vectors in 
$X^*$ of degree $n-3$ to vectors in $U^*$:
\begin{equation}
\label{d-r-dual}
(X^*)^{n-3} ~ \to ~ U^* = (U^*)^{n-2}\,.
\end{equation}

Using the explicit representatives of vectors in $X$ \eqref{summand-more}
and the $S_k$-symmetry of $Y^{\si}_{r_1, \dots, r_k}$, 
we see that the evaluation of every vector  $Y^{\si}_{r_1, \dots, r_k}$ on 
representatives of vectors in $X$ is zero. Thus all elements in $\Xi$ represent 
vectors in $U^*$.  

Due to Claim \ref{cl:forks-are-good}, the cohomology classes of $\Xi$ in 
$H^{n-2}(V_\bullet(n)^* , \de^*_0)$ span a subspace of dimension 
$$
n! - (n-1)!
$$

Thus, since $U^*$ also has dimension $n! - (n-1)!$ and the only 
component of the first potentially non-zero $d^*_r$ is \eqref{d-r-dual}, we conclude that 
$$
\dim \big( E_{\infty} V_\bullet(n) \big) \ge \dim \big( E_2 V_\bullet(n) \big). 
$$
 
Lemma \ref{lem:specseqabutment} is proved. 
\end{proof}

\vspace{1cm}

\noindent\textsc{Department of Mathematics,
Temple University, \\
Wachman Hall Rm. 638\\
1805 N. Broad St.,\\
Philadelphia, PA, 19122 USA \\
\emph{E-mail address:} {\bf vald@temple.edu}}

\vspace{0.5cm}

\noindent\textsc{University of Z\"urich, \\
Institute of Mathematics, \\
Winterthurerstrasse 190,\\
8057 Z\"urich, Switzerland  \\
\emph{E-mail address:} {\bf thomas.willwacher@math.uzh.ch}}

\end{document}